\documentclass[pdflatex,sn-mathphys-num]{sn-jnl}

\usepackage{graphicx}
\usepackage{multirow}%
\usepackage{amsmath,amssymb,amsfonts}%
\usepackage{amsthm}%
\usepackage{mathrsfs}%
\usepackage[title]{appendix}%
\usepackage{xcolor}%
\usepackage{textcomp}%
\usepackage{manyfoot}%
\usepackage{booktabs}%
\usepackage{algorithm}%
\usepackage{algorithmicx}%
\usepackage{algpseudocode}%
\usepackage{listings}%
\usepackage{float} 
\usepackage{cleveref}
\usepackage[labelformat=simple]{subcaption} 
\usepackage{placeins}
\crefname{figure}{Figure}{Figure}
\Crefname{figure}{Figure}{Figure}

\newcommand{\vect}[1]{\boldsymbol{#1}}
\newcommand{\lambar}{\bar{\lambda}}
\newcommand{\phibar}{\bar{\phi}}
\newcommand{\abs}[1]{\lvert #1 \rvert}
\newcommand{\norm}[1]{\left\lVert#1\right\rVert}
\newcommand{\flow}{F}
\newcommand{\koop}{{K}_\flow}
\newcommand{\inner}[2]{\langle #1, #2 \rangle}

\newcommand{\vx}{\vect{x}}
\newcommand{\vy}{\vect{y}}
\newcommand{\vw}{\vect{w}}
\newcommand{\vv}{\vect{v}}
\newcommand{\vq}{\vect{q}}

\theoremstyle{thmstyleone}%
\newtheorem{prop}{Proposition}
\theoremstyle{definition}
\newtheorem{definition}{Definition}
\newtheorem{example}{Example}

\newtheorem{remark}{Remark}

\begin{document}

\title[On the algebra of Koopman eigenfunctions 
and on some of their infinities]{On the algebra of Koopman eigenfunctions \\
and on some of their infinities}


\author[1,2,3]{\fnm{Zahra} \sur{Monfared}}\email{zahra.monfared@iwr.uni-heidelberg.de}

\author[3]{\fnm{Saksham} \sur{Malhotra}}\email{saksham2196@gmail.com}

\author[3]{\fnm{Sekiya} \sur{Hajime}}\email{hajime.sekiya@fernuni-hagen.de}

\author[4]{\fnm{Ioannis} \sur{Kevrekidis}}\email{yannisk@jhu.edu}

\author*[3]{\fnm{Felix} \sur{Dietrich}}\email{felix.dietrich@tum.de}

\affil[1]{Interdisciplinary Center for Scientific Computing, University of Heidelberg, Germany}

\affil[2]{Department of Mathematics and Computer Science, University of Heidelberg, Germany}

\affil[3]{School of Computation, Information and Technology, Technical University of Munich, Germany \& MDSI \& MCML}

\affil[4]{Departments of Chemical and Biomolecular Engineering and of Applied Mathematics and Statistics, Johns Hopkins University, Baltimore, USA}


\abstract{For continuous-time dynamical systems with reversible trajectories, the nowhere-vanishing eigenfunctions of the Koopman operator of the system form a multiplicative group.
Here, we exploit this property to accelerate the systematic numerical computation of the eigenspaces of the operator. Given a small set of (so-called ``principal'') eigenfunctions that are approximated conventionally, 
we can obtain a much larger set by constructing polynomials of the principal eigenfunctions.
This enriches the set, and thus allows us to more accurately represent application-specific observables. 
Often, eigenfunctions exhibit localized singularities (e.g. in simple, one-dimensional problems with multiple steady states) or extended ones (e.g. in simple, two-dimensional problems possessing a limit cycle, or a separatrix); 
we discuss eigenfunction matching/continuation across such singularities.
By handling eigenfunction singularities and enabling their continuation, our approach supports learning consistent global representations from locally sampled data. This is particularly relevant for multistable systems and applications with sparse or fragmented measurements.}




\maketitle

\section{Introduction}
Much research on Koopman operator approximation for dynamical systems focuses on single basins of attraction around fixed points~\cite{mauroy-2013,mauroy2020koopman,kaiser2021data}, something that is exploited, at least locally, in the celebrated Hartman-Grobman theorem~\cite{hartman-1960,grobman-1959}; for more recent results extending this to the entire basin of attraction, see \cite{mezic-2019b,kvalheim-2025}.
This is not surprising, because the spectrum of the Koopman operator around hyperbolic fixed points can be characterized using the relation to the local linearization around the fixed point~\cite{applied-koopmanism:2012}.
However, even seemingly simple systems in one dimension with multiple steady states have complicated Koopman eigenfunction structure. Consider, for example, the following system defined through an ordinary differential equation with a polynomial vector field,
\begin{equation}\label{eq:motivating example system}
    \dot{x}=(x-a)(x-b)(x-c),\ x\in\mathbb{R}.
\end{equation}
Figure~\ref{fig:figure1 1d system} shows the three steady states of the system (\Cref{eq:motivating example system} for the case $a=-1,b=0,c=3$):
It also shows the vector field itself, as well as three Koopman eigenfunctions, whose selection we now discuss.
\begin{figure}[ht]
    \centering
    \includegraphics[width=.9\linewidth]{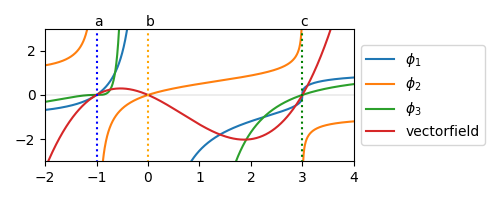}
    \caption{Three Koopman eigenfunctions $\phi_1,\phi_2,\phi_3$ for a system with vector field $\dot{x}=(x-a)(x-b)(x-c)$, with $a=-1$, $b=0$, and $c=3$. The color of the dashed lines at the steady states indicates where the correspondingly colored eigenfunction is zero.}
    \label{fig:figure1 1d system}
\end{figure}
Let us concentrate at the interval between $a$ and $b$. Using the linearization around steady state $a$, and solving the Koopman PDE $\langle\nabla \phi , \dot{x}\rangle=\lambda \phi$, results in the blue eigenfunction $\phi_1$; we can pin, without loss of generality, the derivative of $\phi_1$ at $x=a$ to an arbitrary value; this is just a scaling to select a single eigenfunction among all acceptable scalings.
This eigenfunction helps us describe trajectories of the dynamical system between $a$ and $b$.
The same trajectories between $a$ and $b$ can also be described using the eigenfunction $\phi_2$, similarly constructed by focusing on the steady state $b$ and its linearization eigenvalue. Crucially, the two eigenfunctions are simple transformations of each other: $\phi_1 = \phi_2 ^{\lambda_1/\lambda_2}$. 
This is not surprising, as we will discuss below (\Cref{prop: multiplicative_prop}): Powers of Koopman eigenfunctions are also Koopman eigenfunctions (a well-known fact, see~\cite{applied-koopmanism:2012}).

The second important point to notice is that $\phi_1$ asymptotes to infinity at $b$, while $\phi_2$ asymptotes to infinity at $a$; so, the power relation between the two should only be valid in the open interval $(a,b)$, where they are both finite (for details we refer to the concept of ``open''~\cite{mezic-2019b} or, similarly, ``primary'' eigenfunctions~\cite{bollt-2021a}).
Clearly, $\phi_1$ is also defined in $(-\infty, a)$, while $\phi_2$ is also defined in $(b,c)$. Using the power relation $\phi_1 = \phi_2 ^{\lambda_1/\lambda_2}$ we can now extend (in the spirit of analytic continuation) $\phi_2$ even to the left of $a$, and $\phi_1$ even to the right of $b$, up to $c$, where $\phi_2$ asymptotes to infinity. 
The same process allows us to realize that $\phi_3$ (computed based on the eigenvalue of the steady state $c$ and the solution of the Koopman PDE in the interval $(b,c)$) can be transformed in that same interval to $\phi_2$, using the ratio of $\lambda_2$ and $\lambda_3$ as the requisite power. 
And, finally, the same extension argument can be used to extend each of these eigenfunctions over the entire real line, with the possible exception of the singularities at the various steady states.

In a sense, we can talk about \textit{a single, unique Koopman eigenfunction over the entire domain}. The power transformation we discussed generates the other ones, modulo the (occasional) point-wise singularities.
Both the items we discussed, the power transformation and the singularities, will naturally play a role in computing the eigenfunctions; that is the subject of the paper.
In the first part of the paper we focus on the power relation between eigenfunctions (separating it from the singularities) and only discuss \Cref{prop: multiplicative_prop} and how it can be used to generate new eigenfunctions.


The paper is organized as follows.
In \Cref{sec:intro to computation}, we motivate how the group structure of the Koopman eigenfunctions can be used to systematize and speed up eigencomputations.
In \Cref{sec:related work}, we provide an overview of the Koopman operator framework and review various existing approximation algorithms. \Cref{sec:mathematical framework} focuses on the methodology for extending the computed set of eigenfunctions, with a detailed analysis presented in  \Cref{sec:algorithms for operator eigenfunctions}.
Then, in \Cref{sec:numerical examples}, we illustrate our approach through several numerical examples, showcasing the construction of an expanded set of eigenfunctions. We also discuss how to approximate eigenfunctions with singularities, which is necessary for, e.g., systems with transients to multiple attractors. We further discuss how to consistently extend them across the singularities and how, when possible, to transform between them despite the singularities.
Finally, in \Cref{sec:discussion}, we explore potential extensions and use cases of our work.

\section{Computation of Koopman eigenfunctions}\label{sec:intro to computation}

Numerical computation of the spectral decomposition of the Koopman operator from observation data is an important yet challenging problem~\cite{edmd:2015,li-2017a,schmid-2022,colbrook-2024a}.
The main goal of Koopman operator approximation methods for dynamical systems with pure point spectrum is the accurate representation of a sufficiently large set of eigenfunctions. In the ideal case, the observables of interest lie in the span of these eigenfunctions. Then, prediction of future values of these observables is trivial and only amounts to multiplication.

The notion of ``principal'' eigenfunctions has been proposed~\cite{mezic-2017,bollt-2021a}:
a minimal set of functions, such that the entire eigenspace of the operator can be constructed by systematically combining them (e.g., taking integer-valued powers).

A prominent approximation algorithm, mainly for large-scale linear systems, is the Dynamic Mode Decomposition~\cite{dmd-schmid:2010} (DMD). Its extension~\cite{edmd:2015} 
(EDMD) approximates the spectral properties from the action of the operator on a larger set of functions. Many variants of these algorithms are available, see for example the recent work for ergodic systems~\cite{mpdmd:2023}.
Yet, most of these algorithms only employ the linearity of the operator in order to approximate a matrix representation of it: its own approximate Koopman matrix. 
This means that the chosen dictionary must be expressive enough to represent many Koopman eigenfunctions, so that the action of the operator on the dictionary does not deviate too much from the finite-dimensional dictionary subspace.

In this paper, we propose to employ another property of the Koopman operator, which is not common to all linear operators: the group structure of its eigenfunctions~\cite{mohr-2014}.
Given any sufficiently well-approximated eigenfunction of the operator, (even negative) integer powers of the function are also eigenfunctions. This is true even if these powers cannot be represented in the original dictionary. It is therefore possible to derive eigenfunctions that cannot be well approximated  using the initial Koopman matrix representation. As a motivating example (from~\cite{kaiser2021data}),
consider a nonlinear ODE system of the form
\begin{align}\label{nonlinear-sys}
      \begin{cases}   
        \dot{x_1} = a x_1, \\[1ex]
        \dot{x_2} = b \big(x_2 - x_1^2 \big).
    \end{cases}
\end{align} 
By an appropriate choice of observable functions, this system can be embedded in a higher-dimensional space where the dynamics become linear, although for a general nonlinear system it is not always possible to achieve linearization in finite dimensions. For instance, for $\vy=(y_1, y_2, y_3) = (x_1, x_2, x_1^2)$, \eqref{nonlinear-sys} can be represented as a 3D linear system that remains closed under the action of the Koopman operator. 
However, to better illustrate our point with this example, we expand it using more observable functions $\vy=(y_1, y_2, y_3, y_4, y_5) = (x_1, x_2, x_1^2, x_1 \, x_2, x_1^
3)$ for which system \eqref{nonlinear-sys} will be transformed into a 5D linear system:
\begin{align}\label{5dlinear-sys}
      \begin{cases}   
        \dot{y_1} = a y_1 \\[1ex]
        \dot{y_2} = b \big(y_2 - y_3 \big)
        \\[1ex]
        \dot{y_3} = 2 a y_3 
        \\[1ex]
        \dot{y_4} = (a+b)y_4 - b y_5 
        \\[1ex]
        \dot{y_5} = 3a y_5
    \end{cases}.  
\end{align} 
 The flow map $\flow^{\Delta t}:\mathcal{M} \longrightarrow \mathcal{M} $ is given by 
\begin{align}
 \flow^{\Delta t} (\vy_0) \, = \, \exp{\bigg(\begin{pmatrix}
a & 0 & 0 & 0 & 0 \\
0 & b & -b & 0 & 0 \\
0 & 0 & 2a & 0 & 0 \\
0 & 0 & 0 & a+b & -b \\
0 & 0 & 0 & 0 & 3a
\end{pmatrix} \Delta t \bigg)} \vy_0   \, = \,  \exp{\big(A \, \Delta t \big)} \vy_0  ,
\end{align}
where $ \vy_0 = \vy(0)$. 
The matrix $A$ represents the finite-dimensional Koopman generator associated with the chosen observable vector $\vy$. 
Indeed, the observables evolve according to the linear system
\begin{align}
\dot{\vy}=A \, \vy,
\end{align}
The discrete-time Koopman operator over a sampling interval $\Delta t$ is then given by
\begin{align}
K = e^{A\Delta t}.
\end{align}
If $\vw$ is a left eigenvector of $A$ satisfying
\begin{align}
\vw^T \, A=\lambda \vw^T,
\end{align}
then the function
\begin{align}
\phi(\vy)=\langle \vy, \vw \rangle = \vw^T \, \vy
\end{align}
satisfies
\begin{align}
 \frac{d}{dt}\phi(\vy)=  \langle\nabla \phi , \dot{\vy}\rangle = 
 \vw^T  A \vy  = \lambda\phi(\vy).
\end{align}
Therefore, $\phi(\vy)$ is a Koopman eigenfunction associated with the eigenvalue $\lambda$. \\

The left eigenvectors of the matrix $A$ are
\begin{align}\nonumber
 \vw_1 &\, = \, (1 \, \, \,  0 \, \, \, 0 \, \, \, 0 \, \, \, 0)^T \\\nonumber
 \vw_2 &\, = \, (0 \, \, \, 0 \, \, \, 1 \, \, \, 0 \, \, \, 0)^T  \\\nonumber
 \vw_3 &\, = \, (0 \, \, \, 0 \, \, \, 0 \, \, \, 0 \, \, \, 1)^T \\\nonumber
 \vw_4 &\, = \, (0 \, \, \, \, \frac{2a-b}{b} \, \, \, \, 1 \, \, \, 0 \, \, \, 0)^T \\
 \vw_5 &\, = \, (0 \, \, \, 0 \, \, \, 0 \, \, \, \, \frac{2a-b}{b} \, \, \, \, 1)^T,
\end{align}
and hence the eigenfunctions of the associated Koopman operator are
\begin{align}\nonumber
  \phi_1(\vy) &\, = \, \langle \vy, \vw_1 \rangle = y_1, 
  \\[1ex]\nonumber
  \phi_2(\vy) &\, = \, \langle \vy, \vw_2 \rangle = y_3 = y_1^2 = \phi_1^2(\vy)
  \\[1ex]\nonumber
  \phi_3(\vy) &\, = \, \langle \vy, \vw_3 \rangle = y_5 = y_1^3  = \phi_1^3(\vy)
  \\[1ex]\nonumber
  \phi_4(\vy) &\, = \, \langle \vy, \vw_4 \rangle = \frac{2a-b}{b} y_2 + y_3 = \frac{2a-b}{b} y_2 + y_1^2
  \\[1ex]
  \phi_5(\vy) &\, = \, \langle \vy, \vw_5 \rangle = \frac{2a-b}{b} y_4 + y_5 =  \frac{2a-b}{b} y_1 y_2 + y_1^3.
\end{align}

Thus, the DMD algorithm creates a $5 \times 5$ Koopman matrix where the eigenfunction $\phi_2$ is the square of $\phi_1$ and $\phi_3$ is the cube of $\phi_1$. However, the squares of the functions $\phi_2, \phi_3, \phi_4,\phi_5$ are also eigenfunctions of the Koopman operator; yet, they cannot be derived from the given Koopman matrix (see~\Cref{figure:schematic figure 2}). Similarly, higher powers of the eigenfunctions cannot be derived from this matrix either. 
\begin{figure}[ht!]
    \centering
    \includegraphics[width=.9\textwidth]{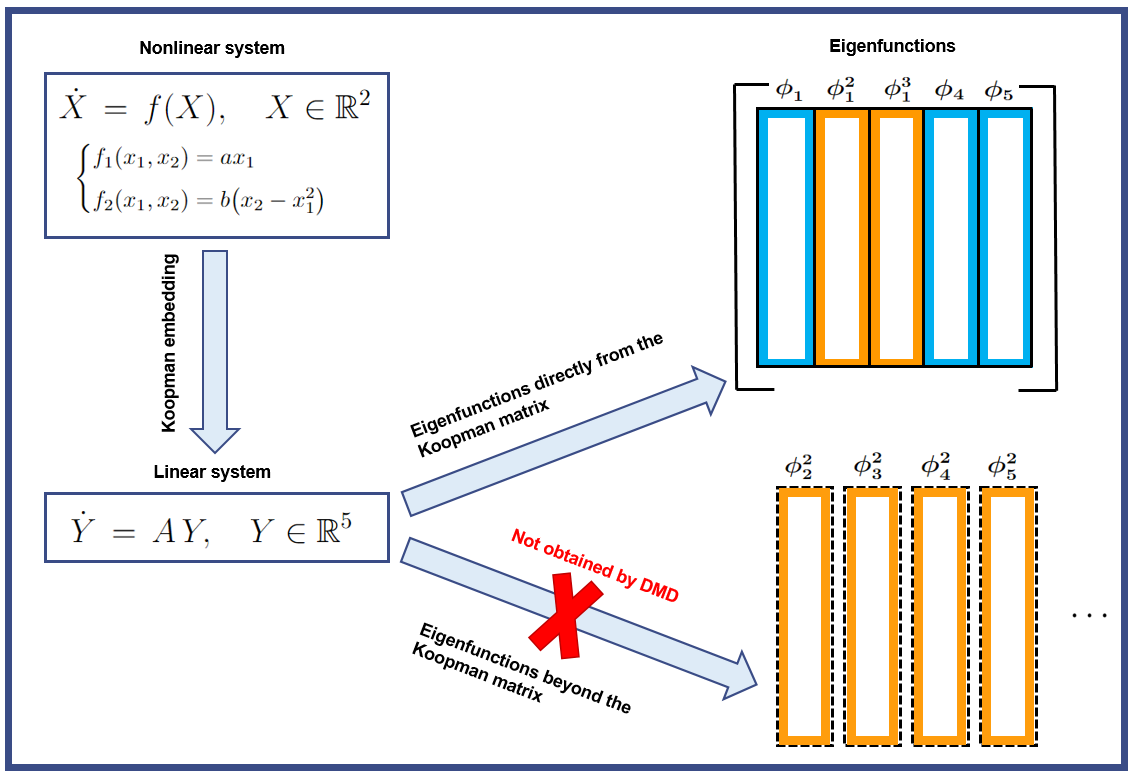} 
    \caption{
    \label{figure:schematic figure 2}Our approach allows for the computation of eigenfunctions that cannot be derived directly from the given Koopman matrix we started with. For system \eqref{nonlinear-sys}, the eigenfunctions $\phi_1, \phi_2, \phi_3, \phi_4, \phi_5$ can be obtained from the Koopman matrix $A$. While the second and third powers of $\phi_1$ are also eigenfunctions that can be derived from this matrix, the powers of $\phi_2, \phi_3, \phi_4$, and $\phi_5$ cannot be obtained directly from it. Therefore, our approach enables the derivation of additional eigenfunctions by utilizing higher powers ($p \geq 4$) of $\phi_1$ and by considering powers $p \geq 2$ of $\phi_2, \phi_3, \phi_4$, and $\phi_5$. Blue lines represent eigenfunctions obtained from the Koopman matrix, while orange lines correspond to integer powers of eigenfunctions. Solid lines indicate eigenfunctions derived directly from the Koopman matrix, whereas dashed lines denote those that are not.}
\end{figure}
\section{Related work}\label{sec:related work}

We now briefly discuss research related to the algebra of Koopman eigenfunctions, and their extension across singularities. The Koopman operator is often studied in the ergodic setting, which we introduce in \Cref{sec:ergodic system}. The considerations on singular eigenfunctions concern the approach to attractors and not the attractors themselves (on which the dynamics can be ergodic), and so we require a more general setting.
In general, a dynamical system is defined by a set $\mathcal{M}$ called the state space or the phase space and a map $\flow: \mathcal{M} \to \mathcal{M}$ called the flow (or evolution function).
Typically, one considers the case where $\mathcal{M}$ is a measurable space, with a $\sigma$-algebra $\mathfrak{B}$, and $\flow$ is $\mathfrak{B}$-measurable (for more information see \cite{applied-koopmanism:2012}).

\subsection{Koopman operator and its multiplicative property}
Let $\mathcal{M}$ be a $\mathfrak{B}$-measurable set and $\mathcal{F}=\left\lbrace f:\mathcal{M}\to\mathbb{C}\right\rbrace$ a space of measurable, complex-valued functions. 
\begin{definition}(Koopman operator (discrete-time));
Consider a discrete-time dynamical system defined by the nonlinear (non-singular) map $F: \mathcal{M} \to \mathcal{M}$ s.t. $\vx({n+1})=F(\vx(n)),\ n\in\mathbb{N}$, $\vx(1)\in \mathcal{M}$. The Koopman operator $\koop:\mathcal{F}\to\mathcal{F}$ associated with
the map $F: \mathcal{M} \to \mathcal{M}$ acts on observables $g\in\mathcal{F}$ and is defined through the composition
\begin{align}
[\koop g](\vx(n))=(g\circ F)(\vx(n)).    
\end{align}
The composition on the left is a linear operation, so $\koop$ is a linear operator on $\mathcal{F}$ and can be spectrally decomposed. 
\end{definition} 
\begin{definition}(Koopman operator (continuous-time));
Consider a continuous-time dynamical system $ \dot{\vx}=F(\vx), \, \vx \in \mathcal{M}$, described by the one-parameter family of (non-singular) flow maps $F^t: \mathcal{M} \to \mathcal{M}, t \in \mathbb{R}^+$. The family of Koopman operators $\mathrm{}{K}^t_{F^t}:\mathcal{F}\to\mathcal{F}$ associated with the family of flow maps $F^t$ acts on observables $g\in \mathcal{F}$ and is defined as
\begin{align}
[\mathrm{}{K}^t_{F^t} \, g](\vx)=(g\circ F^t)(\vx).    
\end{align}
\end{definition} 
\begin{definition}(Koopman eigenfunction and eigenvalue (discrete-time));
An eigenfunction of the Koopman operator $\koop$ associated with the map $F: \mathcal{M} \to \mathcal{M}$ is a nonzero observable $\phi_k \in \mathcal{F} \setminus \{0 \}$ such that 
\begin{align}
\koop \phi_k \, = \, \phi_k \circ F \, = \, \lambda_k \phi_k,    
\end{align}
where $\lambda_k \in \mathbb{C}$ is the corresponding eigenvalue.
\end{definition} 
\begin{definition}(Koopman eigenfunction and eigenvalue (continuous-time));
An eigenfunction of the Koopman operator $\mathrm{}{K}^t_{F^t}$ associated with the semigroup of flow maps $(F^t)_{t \geq 0}$ is a nonzero observable $\phi_k \in \mathcal{F} \setminus \{0 \}$ such that 
\begin{align}
\mathrm{}{K}^t_{F^t} \phi_k \, = \, \phi_k \circ F^t \, = \, e^{\lambda_k t} \phi_k \hspace{.4cm} \forall t \geq 0,   
\end{align}
where $\lambda_k \in \mathbb{C}$ and $e^{\lambda_k t}$ is the corresponding eigenvalue.\\[1ex]
If the semigroup of operators $\mathrm{}{K}^t_{F^t}$ is strongly continuous, Koopman eigenfunctions and eigenvalues are characterized by the relation $\mathcal{L}\phi_k \, = \, \lambda_k \phi_k$. In the case where the semigroup $(F^t)_{t \geq 0}$ arises from the flow of the system $\dot{\vx}=F(\vx)$, they are obtained by solving the following equation \cite{mauroy2020koopman}
\begin{align}
 F \cdot \nabla \phi_{k} =\lambda_{k} \phi_{k}.  
\end{align}
\end{definition} 
A key property of the Koopman operator, central to this work, is its multiplicative structure:
\begin{prop}\label{prop: multiplicative_prop}Products and (integer) powers of eigenfunctions are also eigenfunctions of $\koop$, if the combination is in $\mathcal{F}$:
 $ \koop[\phi_{k_1}^{m_1}\phi_{k_2}^{m_2}]=\lambda_{k_1}^{m_1}\lambda_{k_2}^{m_2}[\phi_{k_1}^{m_1}\phi_{k_2}^{m_2}],$ with $k_1,k_2 \in \mathbb{N}$ and $m_1,m_2 \in \mathbb{Z}$.
 The same property holds for $\mathrm{}{K}^t_{F^t}$, corresponding to continuous-time systems.
\end{prop}
%
\begin{proof}
Straightforward application of the definition of the Koopman operator.
\end{proof}
 Note that \Cref{prop: multiplicative_prop} requires that the powers of the eigenfunctions must be an element of the function space $\mathcal{F}$. This means for negative $m_1$ or $m_2$ the associated eigenfunctions must not be zero on their domain.
In the paper, whenever $|m_1|>1$ or $|m_2|>1$, we will call the corresponding functions {\em monomial eigenfunctions} to distinguish them from {\em principal eigenfunctions}.

 This characteristic imposes {\em a lattice or group structure} on the Koopman operator's spectrum \cite{mauroy2020koopman, applied-koopmanism:2012, lee2023optimized}. Building on this property, in \cite{lee2023optimized} the authors introduced Multiplicative Dynamic Mode Decomposition (MultDMD) to incorporate the Koopman operator's multiplicative structure into its finite-dimensional approximation. They developed a specialized dictionary of basis functions and an efficient optimization algorithm to maintain this structure, ensuring that the computed eigenfunctions include powers. Instead, in this work, we consider an arbitrary dictionary and generate a large set of eigenfunctions by taking their powers, either by multiplying them with themselves or with other eigenfunctions. 
\subsection{Extending eigenfunctions beyond infinity}\label{sec-infinities}
To create the powers of eigenfunctions in order to generate a larger set of them, we need a small set of initial eigenfunctions.  However, approximating these initial eigenfunctions can be quite challenging in some cases. For instance, in certain systems, eigenfunctions may inherently asymptote to $\pm$ infinity on specific subsets of the entire space. 
This can occur, for example, for systems with multiple steady states or limit cycles~\cite{bollt2018matching,mauroy-2013},
where eigenfunctions may diverge at the boundaries between basins of attraction~\cite{bakker2020learning, dietrich2020koopman}. Therefore, it is crucial to develop effective techniques for extending  (analytically continuing) eigenfunctions ``beyond infinity''. 
We will explore how to extend eigenfunctions across singularities for different examples.

\subsection{Existing algorithms for the approximation of the Koopman operator}
Dynamic Mode Decomposition (DMD), which was originally introduced by Schmid and Sesterhenn \cite{dmd-schmid:2010},  can be seen as a variant of a standard Arnoldi method \cite{dmd-arnoldi:2009}. It can be used to approximate the Koopman operator for linear systems~\cite{mezic-2005}. The approximations can also be useful for nonlinear systems, such as fluid flows~\cite{mezic-2013}, if a sufficiently rich set of nonlinear observations of the state of nonlinear system is provided. The latter was formalized as Extended Dynamic Mode Decomposition (EDMD) by Williams, Kevrekidis, and Rowley \cite{edmd:2015}, through the choice of a proper truncated basis of the function space the operator acts on. Li et al.~\cite{li-2017a} use neural networks to construct a flexible, problem-dependent dictionary for this purpose. For measure-preserving dynamical systems, the mpEDMD algorithm as introduced by Colbrook \cite{mpdmd:2023} constrains the spectrum of the approximating matrix to lie on the unit circle.
\subsection{General eigensolvers: Power method and QR algorithm}
We briefly outline the \textbf{power method}, based on Stoer and Bulirsch \cite{intro-numerics:2010}. See Also \Cref{apx:algorithms}. 
Write the eigenpair of $A \in \mathbb{C}^{n\times n}$ as $(\lambda_1, \vv_1), (\lambda_2, \vv_2), \dots (\lambda_n, \vv_n)$ and assume $\lvert \lambda_1 \rvert > \lvert \lambda_2 \rvert \geq \dots \geq \lvert \lambda_n \rvert$. Then for $\vq = \sum_{i=1}^n c_i \vv_i \in \mathbb{C}^n$, it is easy to see that$$
\frac{1}{\lvert \lambda_1 \rvert^m} A^m \vq 
= \sum_{i=1}^n c_i \frac{A^m v_i}{\lvert \lambda_1 \rvert^m}
= \sum_{i=1}^n \left(\frac{\lambda_i}{\lvert \lambda_1 \rvert} \right)^m c_i \vv_i 
\to c_1 \vv_1 (m\to \infty).
$$
Note that $\left(\frac{\lambda_i}{\lvert \lambda_1 \rvert}\right)^m \to 0$ if $i\neq 1$ due to $\frac{\lvert \lambda_i \rvert}{\lvert \lambda_1 \rvert} < 1.$
Therefore, the convergence of this algorithm for any given $\vq$ is linear with respect to the ratio $\frac{\lvert \lambda_2 \rvert}{\lvert \lambda_1 \rvert}$.

The second important algorithm is the \textbf{QR algorithm}. For the initial matrix $A_1$, define a matrix $A_2$ as $A_2 = R_1 Q_1$ where $A_1 = Q_1 R_1$ is the QR decomposition computed by a standard algorithm. By definition, $A_2 = R_1 Q_1 = Q_1^* A_1 Q_1$, and as $Q$ is unitary, $A_1$ and $A_2$ are unitary similar. Thus, they have the same eigenvalues. By repeating this procedure,
$$A_{n+1} = P_n^* A_1 P_n, \text{ where } P_n = Q_1 Q_2 \dots Q_n.$$
Since $P_n$ is a product of unitary matrices, it is again a unitary matrix. Hence, $A_{n+1}$ and $A_1$ are again unitary similar and have the same eigenvalues. $A_{n}$ converges to an upper triangular matrix as $n\to \infty$. 
\section{Mathematical framework}\label{sec:mathematical framework}

\subsection{ Numerical algorithms for Koopman operator approximation}\label{sec:algorithms for operator eigenfunctions}
In this section, we discuss how the algebra of Koopman operator eigenfunctions can enable their computation, enhancing  traditional eigensolvers.

\subsubsection{Constructing eigenfunctions with integer exponents}

Once a single eigenfunction of the Koopman operator has been approximated, the multiplicative property from Proposition \ref{prop: multiplicative_prop} can be used. If $\phibar_1$ and $\phibar_2$ are eigenfunction approximations of the Koopman operator $\koop$ with eigenvalues $\lambar_1$ and $\lambar_2$, then
\begin{align}
    \phibar_{pq} = \phibar_1^p \phibar_2^q,\;\; p,q \in \mathbb{Z} 
\end{align}
is also an eigenfunction approximation of $\koop$ with eigenvalue $\lambar_{pq} = \lambar^p \lambar^q$.
It is interesting at this point to consider negative powers $p,q$. \Cref{fig:infinity_example} shows how a simple eigenfunction $\phi(x)=x$ will lead to a singular eigenfunction if the power $q=-1$ is used. The singular transformation $(\cdot)^q$ induced by taking this \textit{negative} power also provides the simplest demonstration of how continuations across infinity may be rationalized.
\begin{figure}[ht]
    \centering
    \includegraphics[width=0.8\linewidth]{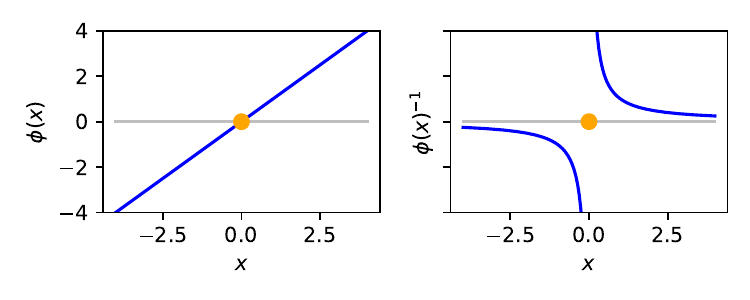}
    \caption{Left: eigenfunction $\phi(x)=x$ for a linear system $\dot{x}=-x.$ Right: inverse of the same eigenfunction, with singularity at the steady state marked in orange.}
    \label{fig:infinity_example}
\end{figure}

In the following, we assume access to data within the basin of attraction of a hyperbolic fixed point of a (possibly nonlinear) dynamical system. In later sections, we will discuss examples with multiple basins, and discuss the related numerical pathologies.
In the single basin case, assume we used EDMD with dictionary $\Psi$ to construct the approximating Koopman matrix $K \in \mathbb{R}^{d \times d}$, with left eigenvector $\{w_i\}_{i=1}^d$ and eigenvalues $\{\lambar_i\}_{i=1}^d$.
Given $i,j \in \{1, \dots, d\}$ and $p, q \in \mathbb{N}\cup\left\lbrace 0\right\rbrace$, the extended set of eigenfunctions is given by
\begin{equation}
    \phibar_{pq}(\vx) = (w_i^T \Psi(\vx))^p (w_j^T \Psi(\vx))^q,
\end{equation}
with corresponding eigenvalues $\lambar_{pq} = \lambar_i^p \lambar_j^q$.
For simplicity, we will focus our analysis only on extending powers of known eigenfunctions and not those of combinations of known eigenfunctions (i.e., we set $q=0$).

\subsubsection{Constructing eigenfunctions with real-valued exponents}

Up to now, we considered integer exponents $q,p$ in constructing new eigenfunctions. 
Yet, the procedure (also Proposition~\ref{prop: multiplicative_prop}) applies to the case of real-valued exponents, and we will take advantage of this later below.
The exponential of the complex number is defined by
$(e^x)^y := e^{y\log(e^x)}, x,y\in \mathbb{C},$
where the complex logarithm of non-zero complex number $z$, denoted as $\log(z)$, is defined by
$e^{\log(z)} := z \in \mathbb{C}.$
Note that if $z$ is given by the polar form $z = re^{i\theta}$ with $r > 0$ and $\theta \in \mathbb{R}$, then the complex logarithm is of the form
$$\log(z) = \ln(r) + i(\theta + 2\pi k), k\in \mathbb{Z},$$
where $\ln(\cdot)$ is the natural logarithm, i.e., $\ln(\cdot) := \log_e(\cdot)$.
Also remember that the principal value of the $\log(z)$ is defined as the logarithm whose imaginary part lies in the $(-\pi, \pi]$, i.e. the principal value is $\ln(r) + i\theta'$ such that $\theta' = \theta + 2\pi k \in (-\pi, \pi], k\in \mathbb{Z}$. In general, $\log(z)$ is taken to denote the principal value.

Now, assume there exists an eigenpair $(\lambda, \varphi)$ with $|\lambda|=1$, $\lambda\neq 1+0i$. Then, one can represent it as $\lambda = e^{ci}$ where $c \in (-\pi, \pi]$. For any $\lambda' = e^{di} \in \mathcal{S}^1$ with $d \in (-\pi, \pi]$, one can consider the power $p := \frac{d}{c} \in \mathbb{R}$ so that 
$$\lambda^p = e^{p\log(\lambda)} = e^{ipc} = e^{i\frac{d}{c}c} = e^{id} = \lambda'.$$
Thus, by \Cref{prop: multiplicative_prop}, $(\lambda', \varphi^p)$ is also an eigenpair.

\subsubsection{Matching eigenfunctions across steady states}
\Cref{prop: multiplicative_prop} provides the mathematical basis for understanding the nature of computationally identified, ``complicated'' monomial eigenfunctions and how they arise from multiplicative combinations of principal ones.
Using the logarithm helps reduce the multiplicative structure and identify the principal eigenfunctions that synthesize it.
This is because for all $\vx\in\mathcal{M}$ where $\phi_{k_1}(\vx)\neq 0$ and $\phi_{k_2}(\vx) \neq 0$,
$$\log|\phi_{k_1}^{m_1}(\vx)\phi_{k_2}^{m_2}(\vx)|=m_1\log|\phi_{k_1}(\vx)|+m_2\log|\phi_{k_2}(\vx)|.$$

Now, consider a (finite) collection of known eigenfunctions $\left\lbrace \phi_k\right\rbrace$, with their domain in a neighborhood $B$ of a particular steady state of $\flow$, s.t. $|\phi_k(\vx)|>0\ \forall \vx\in B$. Note that this is generally the case, because eigenfunctions are typically non-zero away from steady states (unless they are identically zero on a large portion of the state space).

Then, we can systematically filter the collection by removing linear subspaces from it, only leaving the  ``principal spectrum''~\cite{mezic-2017}.
We will see below that computing logarithms will allow us to extend the domain of eigenfunctions far beyond the neighborhood $B$ of the steady state around which they were originally computed; Remarkably, this extension can even ``jump across'' {\em multiple} singularities.

\subsubsection{Algorithms to generate Koopman operator eigenfunctions}
Given a left eigenvalue $\lambda$ and eigenvector $\phi$ of the Koopman operator matrix, we can define an algorithm to find monomial eigenfunctions $\phi_p$ (i.e., $\phi_p \equiv \phi^p$) with prescribed trajectory error/upper bounds.
For a discrete system, there is no time-integration error, so only the error $\delta v$ due to eigenvector approximations plays a role. In this case, given $\epsilon > 0$, to get $E_{\flow G}(\phibar_p, \lambar_p) \leq \epsilon$, we require 
\begin{equation}
    \norm{\delta v} \leq \dfrac{\epsilon^p}{C_{\flow G}(p, \lambar)}.
\end{equation}
Using this relation, Algorithm \ref{alg: extend_discrete_system} shows how to take powers of eigenpairs of the Koopman operator for a discrete time system.

\subsubsection{Error metric for generated eigenfunctions in two dimensions}
Taking powers of eigenfunction approximations accumulates errors. Therefore, we need to define an error function for a monomial eigenfunction. To measure this error in examples with two spatial dimensions, we define a grid $G$ on our domain $\Omega \subseteq \mathbb{R}^2$
\begin{equation}
    \label{eq: grid_G}
    G = \big\{(a + nh, b + mh) \in  \Omega \, \lvert \, a,b \in \mathbb{R}; n, m \in \mathbb{Z}; h \in (0,1) \big\},
\end{equation} 
and a norm $\norm{\cdot}_G$ on the grid
\begin{equation}
    \norm{\phi}_G = \sqrt{\dfrac{1}{\abs{G}}\sum_{\vx \in G} (\phi(\vx))^2}.
\end{equation}
For systems with higher-dimensional states $\vx$, we extend this notion accordingly, to higher-dimensional  grids of equidistant nodes.
We introduce the \textit{trajectory metric} ($E_{\flow G}$) in \Cref{def: traj_error}.
\begin{definition}
\label{def: traj_error}
(Trajectory metric) Given an extended eigenfunction approximation $\phibar_p$ with an extended eigenvalue approximation $\lambar_p$ for the Koopman operator, the trajectory error of the eigenpair approximation $(\lambar_p, \phibar_p)$ is defined, for discrete systems, by
\begin{equation}
    \label{eq: traj_error_discrete}
     E_{\flow G}(\lambar_{p}, \phibar_{p}) =  \norm{ \phibar_{p}(\flow(\cdot) ) - \lambar_{p} \phibar_{p}(\cdot)}_G^{(1/p)},
\end{equation}
 and for continuous systems by
 \begin{equation}
    \label{eq: traj_error_continuous}
     E_{\flow G}(\lambar_{p}, \phibar_{p}) =  \norm{ \phibar_{p}(\flow^{\Delta t} (\cdot) ) - \lambar_{p} \phibar_{p}(\cdot)}_G^{(1/p)}.
\end{equation}
\end{definition}
\begin{algorithm}[htbp]
\caption{Computing extended eigenpairs for discrete system. Given $\norm{\delta w}$, left eigenpair $(\lambar, w_c)$ of the Koopman matrix $K$, $\Psi$ as the dictionary basis and desired trajectory error bound $\epsilon$}
\label{alg: extend_discrete_system}
\begin{algorithmic}[0]
    \State $p \gets 1$
    \While{$p \in \mathbb{N}$}
    \State  $C_{\flow G}(p, \lambar) \gets \norm{ \norm{\Psi(\flow(\cdot)) - \lambar\Psi(\cdot)} \sum_{i=0}^{p-1} \norm{\Psi(\flow(\cdot))}^{p-1-i} \norm{\Psi(\cdot)}^i \lambar^i }_G$
    \If{ \, $ \norm{\delta w} > \dfrac{\epsilon^p}{C_{\flow G}(p, \lambar)} $ \, }
        \State \textbf{break}
    \EndIf
     \State $\phibar_p \gets (w_c^T \Psi)^p$
     \State $\lambar_p = (\lambar)^p$
     \State $p \gets p + 1$
    \EndWhile
\end{algorithmic}
\end{algorithm}

For a continuous system, we also have an integration error in the flow calculation. Given $\epsilon > 0$, to obtain $E_{\flow G}(\phibar_p, \lambar_p) \leq \epsilon$, we require 
\begin{equation}
    \epsilon_G \leq \frac{1}{L}\bigg( \big(\epsilon^p + (\lambar M)^p\big)^{1/p} - \lambar M \bigg).
\end{equation}
Using this relation, Algorithm \ref{alg: extend_cont_system} outlines how to extend eigenpairs of the Koopman operator for a continuous time system.

\begin{algorithm}
\caption{Computing extended eigenpairs for a continuous time system. Given $\epsilon_G$, left eigenpair $(\lambar, w_c)$ of the Koopman matrix $K$ and $\psi$ as the dictionary basis and desired trajectory error bound $\epsilon$}
\label{alg: extend_cont_system}
\begin{algorithmic}[0]
    \State $p \gets 1$
    \While{$p \in \mathbb{N}$}
    \If{\, $ \epsilon_G > \dfrac{1}{L}((\epsilon^p + (\lambar M)^p)^{1/p} - \lambar M) $ \,}
        \State \textbf{break}
    \EndIf
     \State $\phibar_p \gets (w_c^T \psi)^p$
        \State $\lambar_p = (\lambar)^p$
        \State $p \gets p + 1$
    \EndWhile
\end{algorithmic}
\end{algorithm}

\subsubsection{Constructing all admissible eigenvectors for 
continuous time systems}
Given the principal eigenfunctions, \Cref{alg: cont_system_extend_with_bound} describes how to construct all monomial eigenfunctions and corresponding eigenvalues of a continuous time
system that have a bounded error metric.
\begin{algorithm}[ht]
\caption{(Iterative Koopman eigensolver) Algorithm for computing extending eigenvalues $\lambar_p^i$ and extended eigenfunctions $\phibar_p^i$ of a continuous time system, given integration error $\epsilon_G$, desired trajectory error $\epsilon$, the Koopman matrix $K$, $\Psi$ as the dictionary basis and constant $L$}
\label{alg: cont_system_extend_with_bound}
\begin{algorithmic}[0]
\State $i \gets 0$
\State $A \gets K$  
\While{$i < n$}
    \State $\lambar_i, v_i \gets \text{power iteration complex}(A)$
    \State $\lambar_i, w_i \gets \text{power iteration complex}(A^T)$

    \State $p \gets 1$  
    \While{true}
        \If{$ \epsilon_G > \dfrac{1}{L}\Big((\epsilon^p + (\lambar_i M)^p)^{1/p} - \lambar_i M\Big) $}
            \State \textbf{break}
        \EndIf

        \State $\phibar_p^i \gets (w_i^T \Psi)^p$
        \State $\lambar_p^i \gets (\lambar_i)^p$

        \State $p \gets p + 1$  

    \EndWhile
\vspace{.2cm}
    \State $w_i \gets \dfrac{w_i}{w_i^T v_i}$

    \State $A \gets A - \lambar_i v_i w_i^T$

    \If{$\mathrm{imag}(\lambar_i) > 10^{-6}$}

        \State $v_{i+1} \gets \bar{v}_i$
        \State $w_{i+1} \gets \bar{w}_i$
        \State $\lambar_{i+1} \gets \bar{\lambar}_i$
        \State $w_{i+1} \gets \dfrac{w_{i+1}}{w_{i+1}^T v_{i+1}}$

        \State $A \gets A - \lambar_{i+1} v_{i+1} w_{i+1}^T$ 

        \State $i \gets i + 1$

    \EndIf

    \State $i \gets i + 1$

\EndWhile

\end{algorithmic}
\end{algorithm}

\subsection{Theoretical analysis}\label{sec:analysis}

The two main sources of error in eigenfunction approximation are (a) the error in the eigenvector of the Koopman matrix $K$ due to the eigensolver, and (b) the error in the integration for the flow computation in continuous systems. Eigenvector error is present for both discrete and continuous system computations. We can estimate upper bounds for the trajectory error $E_{\flow G}$ with respect to these errors.

\subsubsection{Integration error in continuous systems}
For a continuous time system, the trajectory error will depend on the error introduced while integrating the system to get the flow $\flow^{\Delta t}$. Let
\begin{equation}
    \flow^{\Delta t} (\vx) - \vx^{\Delta t} = \varepsilon(\vx),
\end{equation}
where $\epsilon(\vx) \in \mathbb{R}^n$ is the integration error at $\vx$, $\vx^{\Delta t}$ is the accurate flow at $t = \Delta t$, and $\flow^t(\vx)$  is the computed flow.
We consider the upper bound of the trajectory error with respect to the quantity
\begin{equation}
    \label{eq:epsilon_G}
    \epsilon_G = \max_{\vx \in G} \norm{\varepsilon(\vx)}.
\end{equation}
Using the above, Proposition \ref{prop: upper_bound_continuous_prop} gives the upper bound for eigenfunction approximation $\phibar_p$ and eigenvalue approximation $\lambar_p$ with respect to the integration error $\epsilon_G$.

\begin{prop}
\label{prop: upper_bound_continuous_prop}
\begin{equation}
    E_{FG}(\bar{\phi}_p, \lambar_p) \leq \bigg(\big( \lambar^{\Delta t}  M + L\epsilon_G\big)^p -  (\lambar^{\Delta t} M)^{p}\bigg)^{(1/p)}.
\end{equation}
where $M = \max_{\vx \in G} \norm{\psi(\vx)}$ and $L$ is an upper bound on the spectral norm of the Jacobian of $\Psi$ with respect to the $l_2$ norm on the grid $G$, \begin{equation}
    \norm{J_\Psi (\vx)}_2 \leq L.
\end{equation}
\end{prop}
\begin{proof} 
See \Cref{proof-prop: upper_bound_continuous_prop}
\end{proof}
\begin{remark} To keep the trajectory error below $\epsilon$ for a power $p$, we will require that the integration error $\epsilon_G$ has the bound
\begin{equation}
    \epsilon_G \leq \frac{1}{L} [(\epsilon^p + (\lambar^{\Delta t} M)^p)^{1/p} - \lambar^{\Delta t} M].
\end{equation}
\end{remark}

\begin{remark} To calculate $L$, we compute the maximum singular value of the matrix $J_\Psi(\vx)$ over the 
grid $G$ and take the maximum value
\begin{align}
    \begin{split}
         & \sigma_{\max}(\vx) = \sqrt{\lambda_{\max}\big(J_\Psi(\vx)^T J_\Psi(\vx)\big)}, \\
         & L = \max_{\vx \in G} \sigma_{\max}(\vx), \quad
         L \geq \|J_\Psi(\vx)\|_2 \;\; \forall \vx \in G.
    \end{split}
\end{align}
\end{remark}
 
\subsubsection{Eigenvector approximation error}

Let $\vw$ be a left eigenvector of the Koopman matrix $K \in \mathbb{R}^{d \times d}$. Consider a left eigenvector $\vw_c \in \mathbb{R}^d$ of $K$ computed by an eigensolver. Let
\begin{align*}
    \vw_c = \vw + \delta \vw,
\end{align*}
where $\delta \vw \in \mathbb{R}^d$ is the error vector introduced due to the eigensolver.

We can obtain an \textit{a posteriori} upper bound on the trajectory error in (\ref{eq: traj_error_discrete}) for the $p_{th}$-power eigenfunction approximation
\begin{align}
    \bar{\phi}_p(x) = (\vw_c^T \Psi(\vx))^p.
\end{align}
We assume that the true left eigenvector is normalized so that $\norm{\vw}=1$.
Assuming that $\lambar$ is the computed eigenvalue corresponding to $\vw_c$ and $\lambar_p = \lambar^p$, Proposition \ref{prop: upper_bound_discrete_prop} gives the error bound for the trajectory error with respect to the eigenvector error $\norm{\delta \vw}$.

\begin{prop}
\label{prop: upper_bound_discrete_prop}
\begin{align} 
    E_{FG}(\bar{\phi}_p, \lambar_p) \leq  C_{FG}(p, \lambar)^{(1/p)} \norm{\delta \vw}^{(1/p)},
\end{align}
where
\begin{equation}
    C_{FG}(p, \lambar)= \norm{ \norm{\Psi(F(\cdot)) - \lambar\Psi(\cdot)} \sum_{i=0}^{p-1} \norm{\Psi(F(\cdot))}^{p-1-i} \norm{\Psi(\cdot)}^i \lambar^i }_G.
\end{equation}
\begin{proof} 
See \Cref{proof-prop: upper_bound_discrete_prop}
\end{proof}

\end{prop}

\begin{remark} To keep the trajectory error below $\epsilon$ for a power $p$ we will require that the error in computed eigenvector $\delta \vw$ has norm such that
\begin{equation}
    \norm{\delta \vw} \leq \dfrac{\epsilon^p}{C_{FG}(p, \lambar)}.
\end{equation}
\end{remark}
%
\section{Computational experiments}\label{sec:numerical examples}
We now demonstrate the construction of additional (monomial) eigenfunctions in a series of computational experiments.
In \Cref{sec:example - linear system 2d}, we discuss how Koopman eigenfunctions of a linear system are constructed up to a certain accuracy.
Note that even though the system matrix of the linear system is finite-dimensional and thus only has a finite number of eigenvectors, the Koopman operator of the system still acts on an infinite-dimensional space and has an infinite number of eigenfunctions.
The benefit of analyzing it for a linear system is that all eigenfunctions are available analytically, so we can compute the approximation error of our numerical procedure to the ground truth.
In order {\em to validate the procedure for nonlinear systems}, in~\Cref{sec:example - nonlinear system 2d}, we construct an explicit, nonlinear example by transforming the state space of the linear system in \Cref{sec:example - linear system 2d} with a nonlinear diffeomorphism. This allows us to obtain analytic formulas for Koopman eigenfunctions even in this nonlinear setting.
\Cref{sec:example - separatrix} discusses what happens in general for systems with separatrices.
\Cref{sec:example - isochrons of limit cycles} shows the concept of isochrons for limit cycles.
\Cref{sec:example - saddle isochrons} shows how the concept of {\em isochrons} connects with/relates to the computation of eigenfunctions for systems with separatrices associated with saddle points.

\subsection{Example - linear system in 2D, non ergodic}\label{sec:example - linear system 2d}

Consider the continuous time linear system
\begin{equation}
    \label{eq: continuous_linear_system_example}
    \dot{\vect{\vx}} = A\vect{\vx},
\end{equation}
where $\vx \in \mathbb{R}^{2}$, $A \in \mathbb{R}^{2\times 2}$ with left eigenpairs $(\lambda_1, \vw_1)$ and $(\lambda_2, \vw_2)$.
Let the system be sampled with fixed sampling interval $\Delta t$. Then the Koopman eigenfunctions and eigenvalues of the system are given by
\begin{equation}
 \lambda_{p} = (e^{\lambda_1 \Delta t})^p,\ 
 \phi_{p}(\vx) = \big(\inner{\vw_1}{\vx}\big)^p,\ 
 \lambda_q = (e^{\lambda_2 \Delta t})^q,\ 
 \phi_q(\vx) = \big(\inner{\vect{w}_2}{\vx}\big)^q.
\end{equation}
\noindent
If $\lambar$ is a computed eigenvalue of the Koopman matrix $K$ and $\lambda$ is an eigenvalue of the Koopman generator of the continuous system, then $\lambar \approx e^{\lambda \Delta t}$ where $\Delta t$ is the temporal sampling interval for the system.

We can compare the numerically constructed monomial eigenfunctions with the true eigenfunctions of the system. Let $\phibar_p$ be a constructed, monomial eigenfunction and $\phi_p$ be the true eigenfunction. Then on a grid G as defined in (\ref{eq: grid_G}), the error given by $E_G$ is defined similarly to the discrete system:
\begin{align}
    & G_0 = \{\vect{x}\in G |\;  \abs{\phi_p} < \varepsilon\} \cup \{\vect{x} \in G |\; \abs{\phibar_p} < \varepsilon \}\\
    & c_{mode} = \textit{mode}\bigg(\dfrac{\phi_{p|G/G_0}}{\phibar_{p|G/G_0}}\bigg) \\
    & E_G(\phi_p, \phibar_p) = \norm{\phi_p - c_{mode}\phibar_p}_{G}^{1/p},
\end{align} 
where $\varepsilon$ is some tolerance.
\noindent
The trajectory error for the computed eigenpair $(\lambar_p, \phibar_p)$ is given by
\begin{equation}
    \label{eq:traj_error_cont}
  E_{FG}(\lambar_p, \phibar_p) :=  \norm{ \phibar_p(F^{\Delta t} (\cdot) ) - \lambar_p\phibar_p(\cdot)}_G^{1/p}.
\end{equation}
\noindent Consider the case
\begin{equation}
    A = \begin{bmatrix}
    -0.9 & 0.1\\
    0 & -0.8\\
    \end{bmatrix}.
\end{equation}
The matrix has left eigenpairs $(-0.9, [1, -1]^T)$ and $(-0.8, [0, 1]^T)$. The system is sampled with sampling interval $\Delta t$. Therefore, the true Koopman eigenfunctions of the system are given by
\begin{align}
     \phi_p(\vx) & = \bigg(\frac{x_1- x_2}{\sqrt{2}}\bigg)^p\\
     \phi_q(\vx) & = x_2^q,
\end{align}
with eigenvalues 
\begin{align}
    \lambda_p &= (e^{-0.9 \Delta t})^p\\
    \lambda_q &= (e^{-0.8 \Delta t})^q.
\end{align}

To approximate the eigenfunctions using DMD we collect 400 snapshot pairs (with the state variables as our Koopman observables), where the initial conditions are uniformly randomly distributed between $[-2,2] \times [-2,2]$, with temporal sampling interval $\Delta t=0.2$.
Figure \ref{fig: linear_continuous_DMD} shows the results of the DMD approximation. The out-of-sample prediction shows that the DMD approximation is able to predict trajectories accurately.
\begin{figure}[!htbp]
    \centering
    \includegraphics[scale=.55]{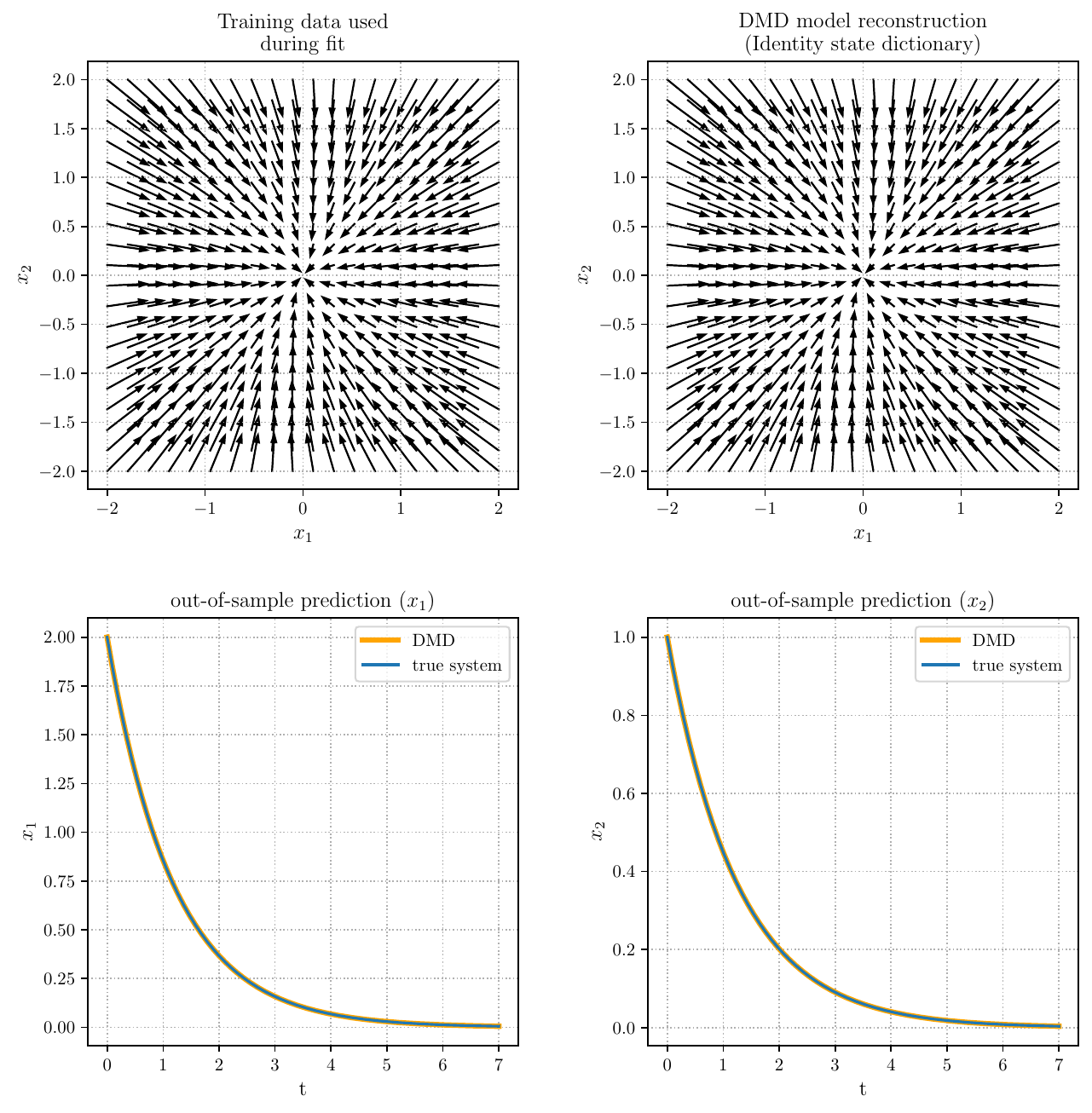}
    \caption{DMD model for the continuous linear system (\ref{eq: continuous_linear_system_example}). Top left shows the training data used for approximation, top right shows the reconstructed data, bottom shows the comparison between predicted trajectory and actual trajectory for a sample initial condition.}
    \label{fig: linear_continuous_DMD}
\end{figure}

We now create a grid $G$ with $a=-1, b=1, n=100, h=0.01$. On $G$, we calculate the flow approximated by the forward Euler method using step size $h=0.001$. Then, we calculate $\epsilon_G$ using (\ref{eq:epsilon_G}) and the true solution,
$
\vx^{\Delta t} = e^{A\Delta t}\vx,
$
and approximated solution using the forward Euler method,
\begin{equation*}
    \vx^{(0)} = \vx,\ 
    N  = \frac{\Delta t}{h},\ 
    \vx^{(i+1)}  = \vx^{(i)} + h A\vx^{(i)}, \ i=0,\dots,N,\ 
    T^{\Delta t}(\vx) = \vx^{(N)}.
\end{equation*}
We then compute the trajectory error for increasing powers, $p$ and $q$ and their upper bound. For DMD, as $\Psi = I$, upper bound, L for $\norm{J_\Psi(\vx)}_2 \leq L$ where L = 1. Figure \ref{fig: linear_continuous_DMD_traj_error_integration_bound} shows the trajectory error and upper bound with respect to the Euler integration error for extended eigenfunctions $\phibar_p$ and $\phibar_q$.

Assuming that flow $F^{\Delta t}(\vx) = e^{A\Delta t}\vx$ we can calculate the trajectory error with respect to eigenvector error by adding a random error vector $\delta v$ with $\norm{\delta v} = 10^{-6}$. Figure \ref{fig: linear_continuous_DMD_traj_error_bound_vector } shows the trajectory error and the upper bound of the trajectory error due to this error in the eigenvector for extended eigenfunctions $\phibar_p$ and $\phibar_q$.
\begin{figure}[!htbp]
    \centering
    \includegraphics[scale=.55]{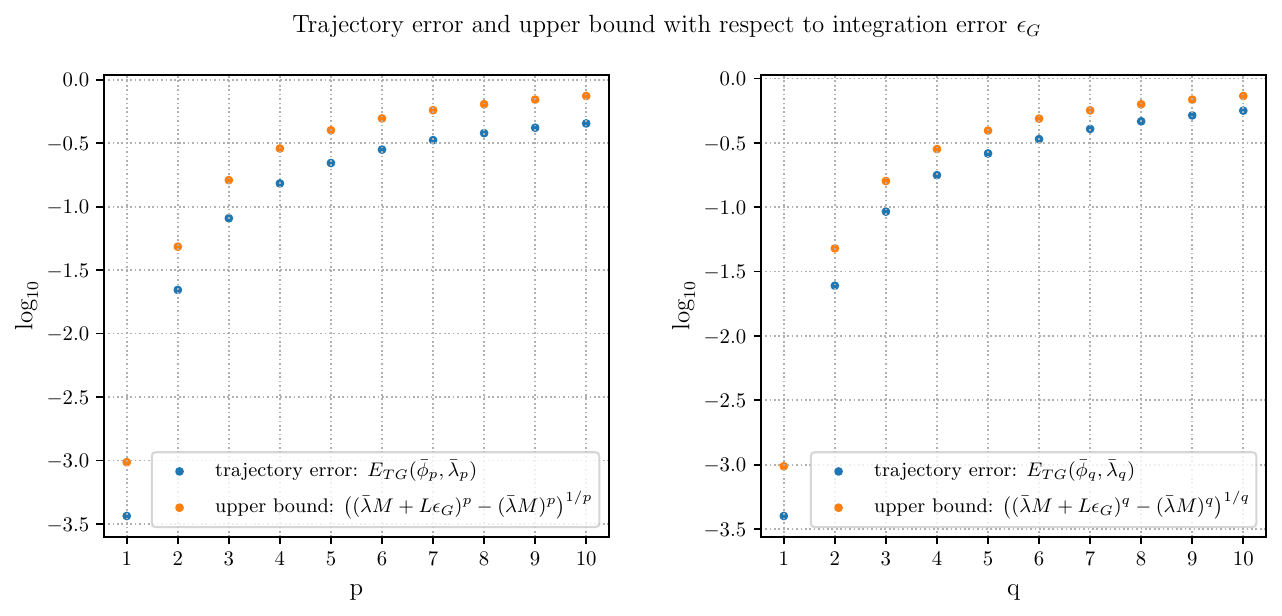}
    \caption{Integration error analysis for DMD eigenfunctions of continuous linear system (\ref{eq: continuous_linear_system_example}). The trajectory error for computed eigenfunctions (blue) and upper bound (orange) with respect to the Euler integration error $\epsilon_G$ for powers $p$ and $q$.}
    \label{fig: linear_continuous_DMD_traj_error_integration_bound}
\end{figure}

Finally, we employ Algorithm \ref{alg: extend_cont_system}, computing trajectory errors and upper bounds with respect to integration error to construct monomial eigenfunctions $\phibar_p$ and $\phibar_q$ with powers $p$ and $q$ such that the trajectory error stays below a desired upper bound $\epsilon = 0.2$. Figure \ref{fig: linear_continuous_DMD_algorithm } shows the results of Algorithm \ref{alg: extend_cont_system}  for a desired trajectory error upper bound $\epsilon = 0.1$. As seen in the figure, the powers $p$ and $q$ suggested by the algorithm are close to the actual powers up to which the eigenfunctions can be extended.
\begin{figure}[!htbp]
    \centering
    \includegraphics[scale=.55]{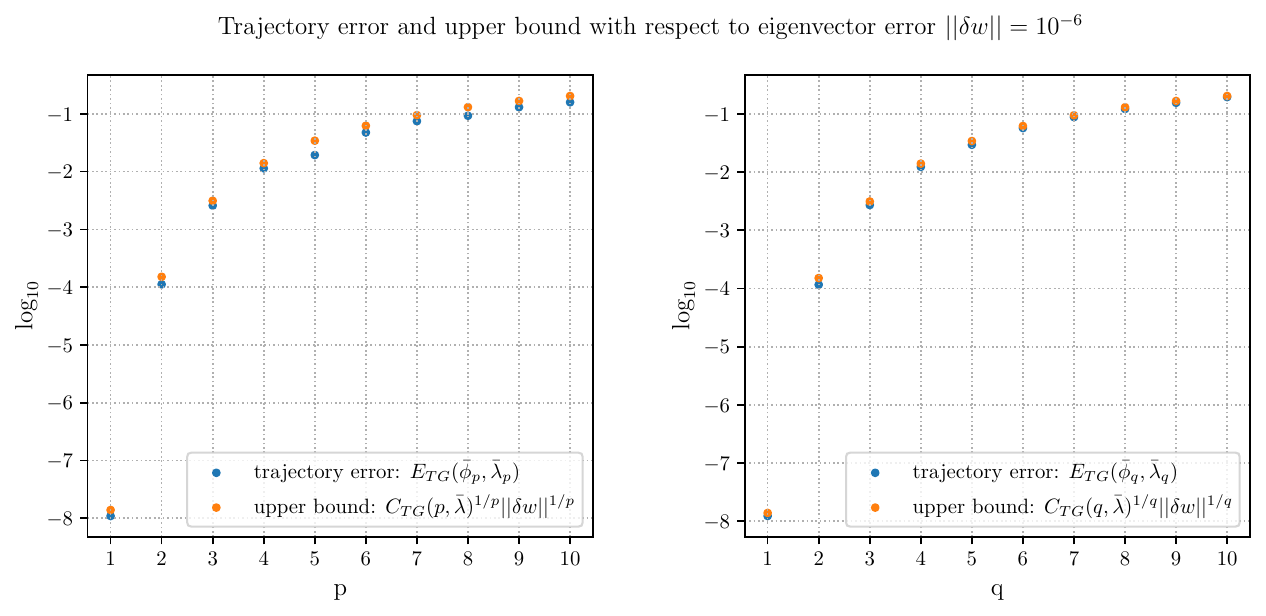}
    \caption{Eigenvector error analysis for DMD eigenfunctions of continuous linear system (\ref{eq: continuous_linear_system_example}). The trajectory error for computed eigenfunctions (blue) and upper bound (orange) with respect to the eigenvector error $\norm{\delta \vw} = 10^{-6}$ for powers $p$ and $q$.}
    \label{fig: linear_continuous_DMD_traj_error_bound_vector }
\end{figure}
\begin{figure}[!htbp]
    \centering
    \includegraphics[scale=.5]{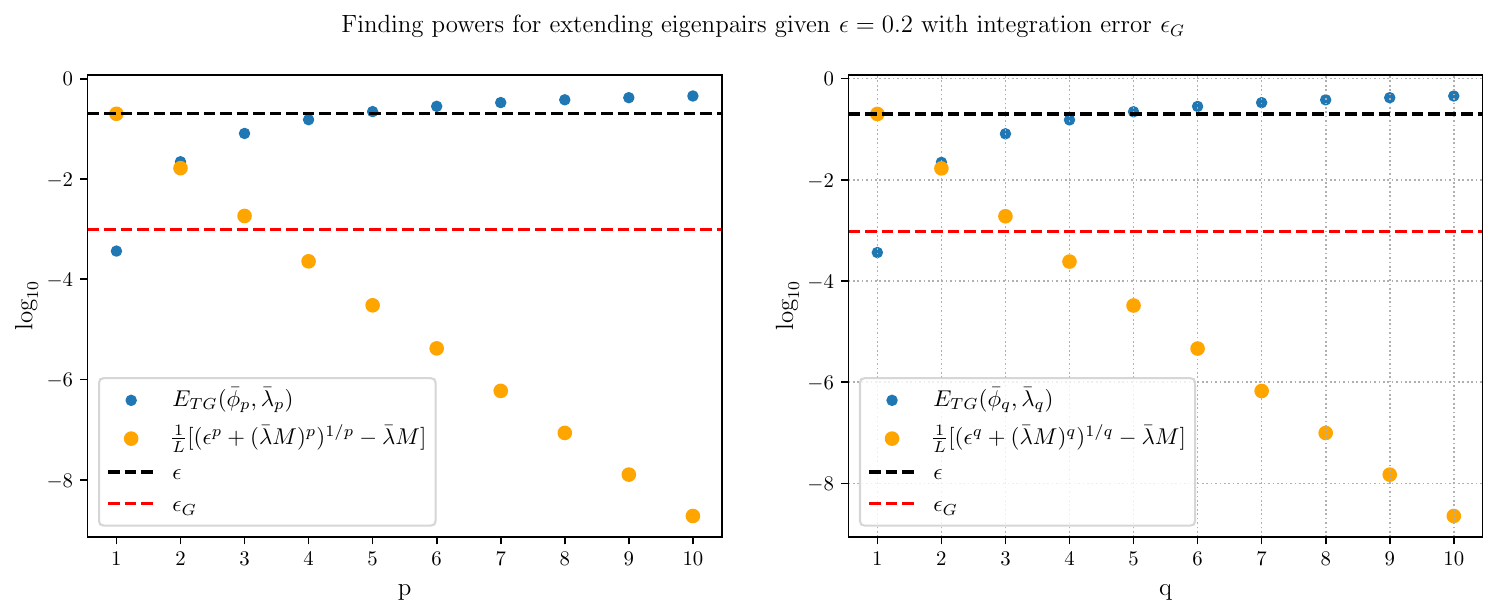}
     \caption{Results of Algorithm \ref{alg: extend_cont_system} applied to the DMD approximation of continuous linear system (\ref{eq: continuous_linear_system_example}) with Euler integration error $\epsilon_G$ and desired trajectory error $\epsilon = 0.2$. The value of $p$ and $q$ suggested by the algorithm-- where the upper bound for $\epsilon_G$ (orange) crosses $\epsilon_G$ (red line) is close to actual value of $p$ and $q$ where the trajectory error (blue) crosses the required $\epsilon$ (black).}
    \label{fig: linear_continuous_DMD_algorithm }
\end{figure}
\subsection{Example - Nonlinear (transformation of linear) system}\label{sec:example - nonlinear system 2d}

Consider the linear continuous system 
\begin{align}
    \label{eq: linear_system_example_2}
    \dot{\vx} = A\vx,
\end{align}
where $A = \begin{bmatrix}
    -0.9 & 0.1\\
    0 & -0.8\\
\end{bmatrix}$.
\\

We transform using the diffeomorphism $\vy = h(\vx) = \log(e^{\vx} + 1)$ (functions applied coordinate-wise) to obtain the new system 
\begin{align}
    \dot{\vy} =\begin{bmatrix}
                1-e^{-y_1} & 0\\
                0 & 1-e^{-y_2}
                \end{bmatrix} A \log(e^{\vy} -1)\label{eq: nonlin_transformed_eqn}.
\end{align}
Using Proposition \ref{prop: multiplicative_prop}, the explicit eigenfunctions of this nonlinear system are then given by ($\phi \circ h^{-1}$) with eigenvalue $\lambda$, where $\phi$ is an eigenfunction of the linear system (\ref{eq: linear_system_example_2}) with eigenvalue $\lambda$.

As the eigenfunctions of the linear system are given by $\phi_i(\vx) = \inner{\vw_i}{\vx}$ with eigenvalues $\lambda_i$,  $(i=1,2)$ where $\vw_i$ is left eigenvector of $A$ with eigenvalue $\lambda_i$, the eigenfunctions of system (\ref{eq: nonlin_transformed_eqn}) are given by 
\begin{align}
    \phi^{nonlin}_i(\vy) = \phi_i\circ h^{-1}(\vy) = \inner{\vw_i}{\log(e^{\vy} -1)}
\end{align}
with eigenvalues $\lambda_i$ for $i=1,2$.

We sample the linear system using the exact solution with $\Delta t = 0.02$, collecting 400 snapshot pairs with initial conditions uniformly randomly distributed between $[-2,2]  \times [-2,2]$. Then, we transform the sampled data using the diffeomorphism $h$. We use this transformed data to perform EDMD with a radial basis function (RBF) dictionary with 40 RBF Gaussian kernel functions with centers calculated using k-means clustering of the transformed data.

We take the grid $G$ with $a=1, b=2, n=100, h=0.01$. Some of the explicit eigenfunctions on the grid $G$ are shown in Figure \ref{fig: explicit_eigenfunctions_nonlin_from_lin} (Appendix \ref{appendix:figures_NTLS}), and the computed EDMD eigenfunctions are shown in Figure \ref{fig: edmd_eigenfunctions_nonlin_from_lin} (Appendix \ref{appendix:figures_NTLS}). Then we calculate $\epsilon_G$ by integrating over the grid and using the integration method \textit{RK45} with the explicit system (\ref{eq: nonlin_transformed_eqn})~\cite{dormand-1980}.

We calculate the first nine eigenpairs of the Koopman matrix.
Then, we use the Koopman eigensolver algorithm defined in Algorithm \ref{alg: cont_system_extend_with_bound} to extend the eigenfunctions of the system. The desired trajectory error is set to $\epsilon = 0.01$. We use the algorithm to get up to $p=3$ extended eigenfunctions for the first nine eigenfunctions with trajectory error less than $\epsilon$. Figure \ref{fig: spectrum_algorithm_extend} shows the spectrum and the powers up to which each eigenpair can be safely extended. Figure \ref{fig: algorithm_extended_eigenfunctions_diffeo_system} (Appendix \ref{appendix:figures_NTLS}) shows some of the extended eigenfunctions. The extended eigenfunctions do not match the explicit eigenfunctions in this case. This might be expected as the EDMD eigenfunctions can differ from the explicit eigenfunctions, as the transformed non-linear system can have an infinite number of independent eigenfunctions; we feel that given the size of the matrix and the accuracy of the calculation, such results are not unreasonable.
\begin{figure}[!htbp]
    \centering
    \includegraphics[scale=0.5]{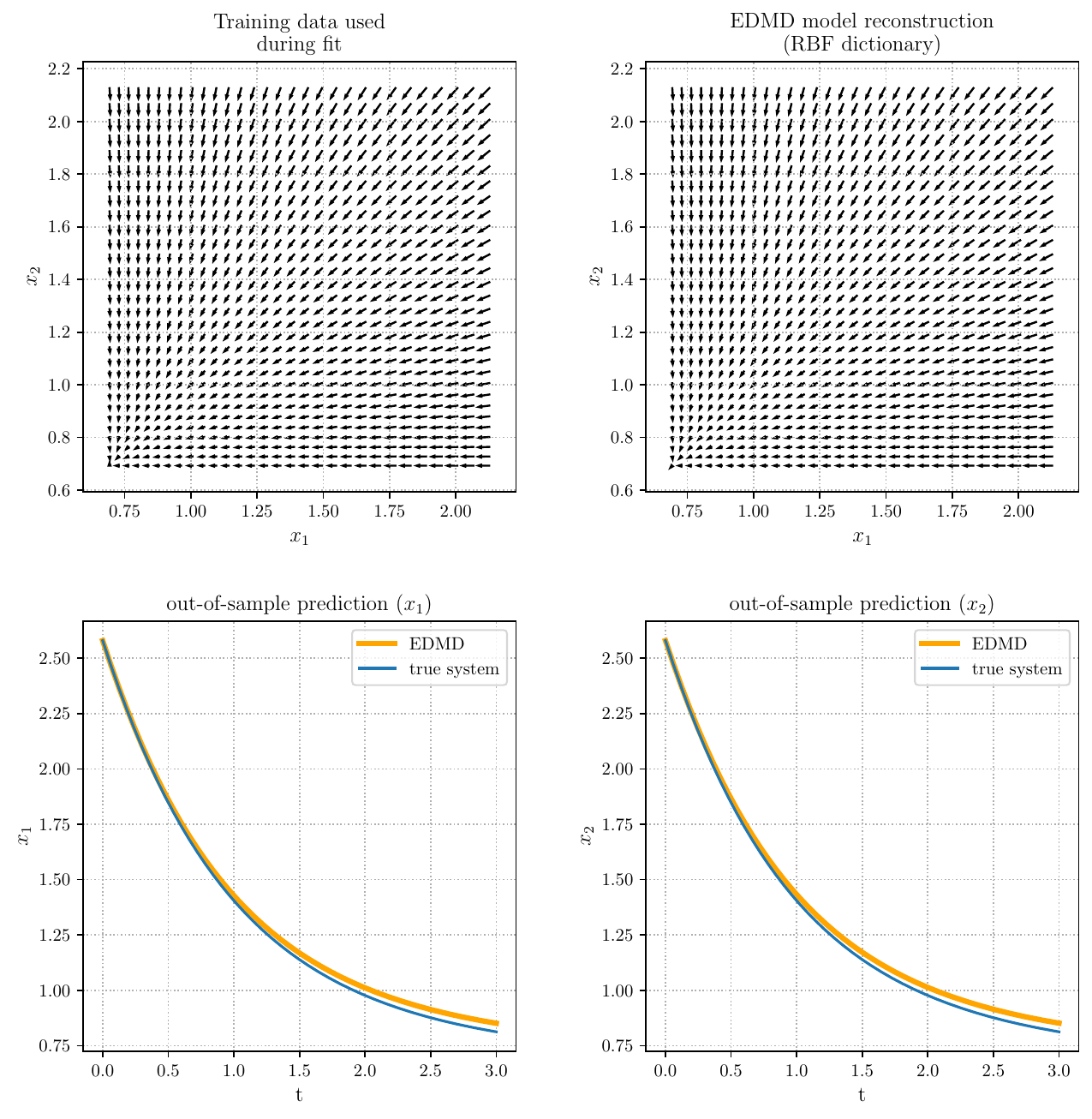}
    \caption{EDMD model for the non-linear system (\ref{eq: nonlin_transformed_eqn}). Top left shows the training data used for approximation, top right shows the reconstructed data, bottom shows the comparison between predicted trajectory and actual trajectory for one initial condition.}
\end{figure}
\begin{figure}[!htbp]
    \centering
 \includegraphics[scale=0.46]{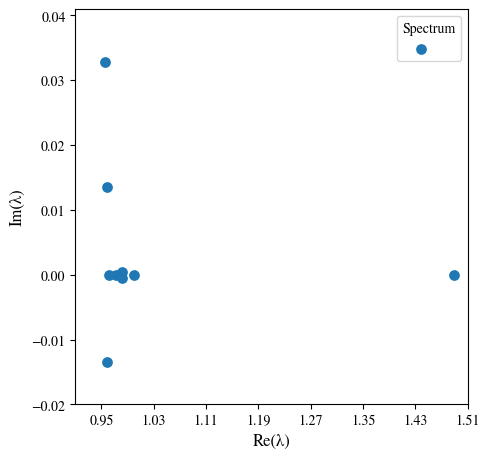}
\hspace{.6cm}\includegraphics[scale=0.46]{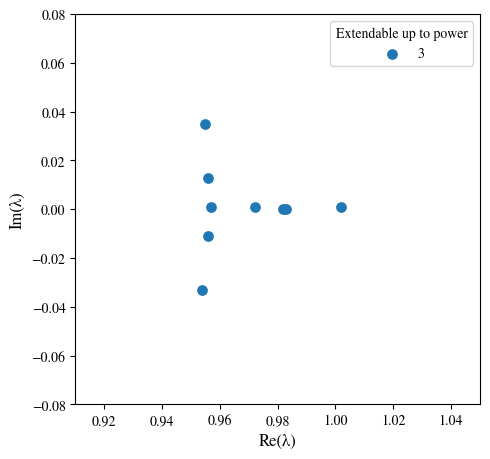}
    \caption{
    %
    %
    %
    Spectrum computed using Algorithm  \ref{alg: cont_system_extend_with_bound} for the nonlinear system (\ref{eq: nonlin_transformed_eqn}) and powers up to which the eigenfunctions can be extended for the first 9 eigenvalues.}
    \label{fig: spectrum_algorithm_extend}
\end{figure}

\subsection{Bridging eigenfunctions across singularities}
To generate new eigenfunctions by repeatedly multiplying them with themselves or with other approximated eigenfunctions, we need to approximate a small initial set of them. However, as discussed in Sect. \ref{sec-infinities}, in some cases, obtaining such initial sets of eigenfunctions is challenging, especially in cases where they involve infinities. In the following examples, we will explore how to extend eigenfunctions beyond infinity for systems with (possibly multiple) singularities.
\subsubsection{Nonlinear system with two steady states}\label{sec:example - separatrix}

Consider an ODE on $\mathbb{R}$, with steady states at $a=2$, $b=3$, s.t.
\begin{align}\label{system-1d}
\dot{x}=(x-a)(x-b).    
\end{align}
The eigenvalues of the linearization around $a$ and $b$ are $(a-b)$ and $(b-a)$, respectively.
Since scalar multiples $c\phi(x)$  of Koopman eigenfunctions $\phi(x)$ are also eigenfunctions, we are allowed to consistently select a single representative member of this family by ``pinning'' the slope of an eigenfunction at some convenient reference value (e.g., link it to the slope of the eigenvector of the linearization at that steady state that has the same eigenvalue).
%
The Koopman eigenfunctions of this system associated with the steady states are obtained by solving $\nabla \phi_{k_i} \cdot \dot{x}=\lambda_{k_i} \phi_{k_i}, \,  i=1,2$. This results in
\begin{align}\label{eigenfunctions_1d}
\phi_{k_1}(x)=\left|\frac{x-a}{x-b}\right|^{\lambda_{k_1} / (b-a)},\hspace{.3cm} \text{and} \hspace{.6cm} \phi_{k_2}(x)=\left|\frac{x-b}{x-a}\right|^{\lambda_{k_2} / (b-a)}.
\end{align}
When $\, \lambda_{k}= k(b-a) \,$ with $k \in \mathbb{Z}$, the eigenfunctions \eqref{eigenfunctions_1d} take the form
\begin{align}\label{eigenfunctions_1d_real}
\phi_{k_1}(x)=\bigg( \frac{x-a}{x-b}\bigg ) ^{k_1},\hspace{.3cm} \text{and} \hspace{.6cm} \phi_{k_2}(x)=\bigg( \frac{x-b}{x-a}\bigg )^{k_2}, \hspace{.6cm} 
 k_1, k_2 \in \mathbb{Z}.    
\end{align}
Then, almost by construction, real-valued Koopman eigenfunctions, 
 associated with one of the steady states become zero at that steady state and approach infinity at the other. For instance, for $k_1 =k_2 = 1$, the Koopman eigenfunctions \eqref{eigenfunctions_1d_real} are plotted in Fig.~\ref{fig_eigenfunctions_real}.
\begin{figure}[H]
\centering
\includegraphics[scale=.64]{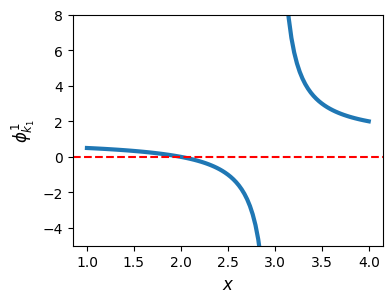}
\hspace{.1cm}
\includegraphics[scale=.64]{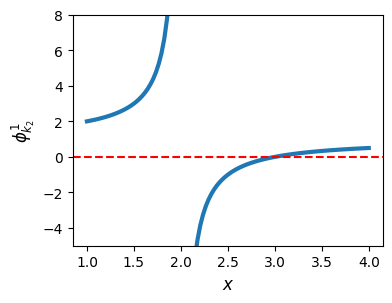}
\caption{The Koopman eigenfunctions \eqref{eigenfunctions_1d_real} associated with the steady states $x=2$ (left) and $x=3$ (right) for $k_1 =k_2 = 1$. As shown, $\phi_{k_1}$ becomes zero at $x=2$ and approaches infinity at $x=3$, while $\phi_{k_2}$ becomes zero at $x=3$ and approaches infinity at $x=2$. }\label{fig_eigenfunctions_real}
\end{figure}
Eigenfunctions of this system can be systematically extended beyond the ``next adjacent'' steady state and even beyond the ``next once removed" steady state. In the interval between
$a$ and $b$, {\em both} steady states lead to properly defined eigenfunction values; we show how this can be exploited to ``cross'' the singularity of the individual eigenfunctions. We first approximate the eigenfunctions associated with the linearization of each steady state in its neighborhood (so, either in a region $[a-\epsilon, a+\epsilon]$, or $[b-\epsilon, b+\epsilon]$, for an $\epsilon<b-a$). If we compute them in a region that is large enough, i.e., we choose $\epsilon$ large, we get approximations in a region $U\subset [a,b]$ for both sets.
At this point, it is appropriate to discuss some observations from computational experiments. 
Clearly, an eigencomputation of the EDMD matrix may converge to any power of the principal eigenvectors (discretized eigenfunctions). In principle, all eigenfunctions, including powers of eigenfunctions corresponding to each of the alternative nearby steady states, would be computable from data if we had an appropriate dictionary. What we practically observe is that data-driven computations with data collected {\em only in the neighborhood of one steady state} tend to numerically converge {\em to the principal eigenfunctions ``corresponding to" that steady state} (not even to their integer powers). We \textit{never} (in our experiments) saw the procedure converge to eigenfunctions associated with the other steady state.
In fact, it is quite challenging to numerically approximate the functions $\phi_{k_1}$ and $\phi_{k_2}$ when they approach their farther away, more remote,``unassociated" steady state, given that they approach infinity quite rapidly there.
%
We compute the logarithm of all functions in a region that excludes points close to the steady states, and test whether they are linearly dependent. The test consists of computing principal components of the logarithm of all eigenfunctions (evaluated on points in $[1,4]$), stacked together in one large dataset $L$: 
\begin{align*}
L=\left[\begin{matrix}
|&|& &|&|&|& &|\\
\log|\phi_{k_1}|&\log|\phi_{k_1}^2|&\cdots&\log|\phi_{k_1}^5| &
\log|\phi_{k_2}|&\log|\phi_{k_2}^2|&\cdots&\log|\phi_{k_2}^5|\\
|&|& &|&|&|& &|
\end{matrix}\right].
\end{align*}
Fig.~\ref{fig:LogAbsKEF_1Dperfect_PCA} (left) shows that {\em only a single direction} is present in this dataset, i.e. all logarithms of the eigenfunctions are linearly dependent. The corresponding principal component of $L$ is also shown (Fig.~\ref{fig:LogAbsKEF_1Dperfect_PCA}, right). Note that it is computed in log space, i.e. negative values here mean very small (but positive) values when considered in the original space.
We cannot compute this with EDMD yet, as we show below. One of the goals of this paper is to discuss how such a global eigenfunction could be constructed numerically from partial EDMD results in the neighborhood of different steady states.

\begin{figure}[ht!]
\centering
\includegraphics[width=.48\textwidth]{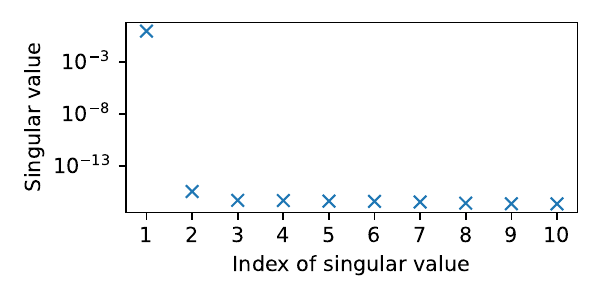}
\includegraphics[width=.48\textwidth]{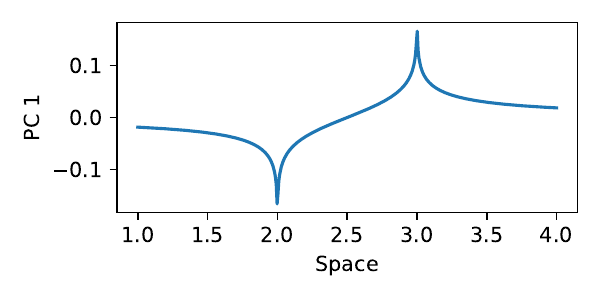}
\caption{\label{fig:LogAbsKEF_1Dperfect_PCA}Left: Energy (singular values) of the principal components of the logarithm dataset. Only one component is relevant, as expected---all the others have numerically zero energy. Right: the principal component associated to the largest singular value.}
\end{figure}
\paragraph{Approximation with EDMD.}
In practice, explicit forms of eigenfunctions are usually not available, and must be approximated from data. We thus now show how to approximate the eigenfunctions numerically, with EDMD, using a radial basis function dictionary (Gaussian kernels), centered at points in the region $[1, 4]$.
We already discussed that data in the neighborhood of each steady state tend to lead to computationally identified eigenfunctions ``corresponding'' to the linearization around that steady state.
Computations combining data across the entire interval tend to provide very poor eigenfunction approximations due to the limitations of the dictionary. We therefore  separately approximate eigenfunctions around $x=a=2$ and $x=b=3$, and then try to map them {\em to each other} in the intermediate region $[a,b]$.
In practice, we cannot accurately approximate eigenfunctions close to their singularities. We thus only work in the regime where their values are comparatively small. 
In the regime where one eigenfunction approaches infinity, we can instead work with the eigenfunction that is its appropriate inverse power - which then approaches zero.
Fig.~\ref{fig:LogAbsKEF_1Dapprox_KEFleft} shows the approximations. Several spurious eigenfunctions are visible (shown are three in the last three columns), which are incorrectly identified.
\begin{figure}[ht!]
\centering
\includegraphics[width=1\textwidth]{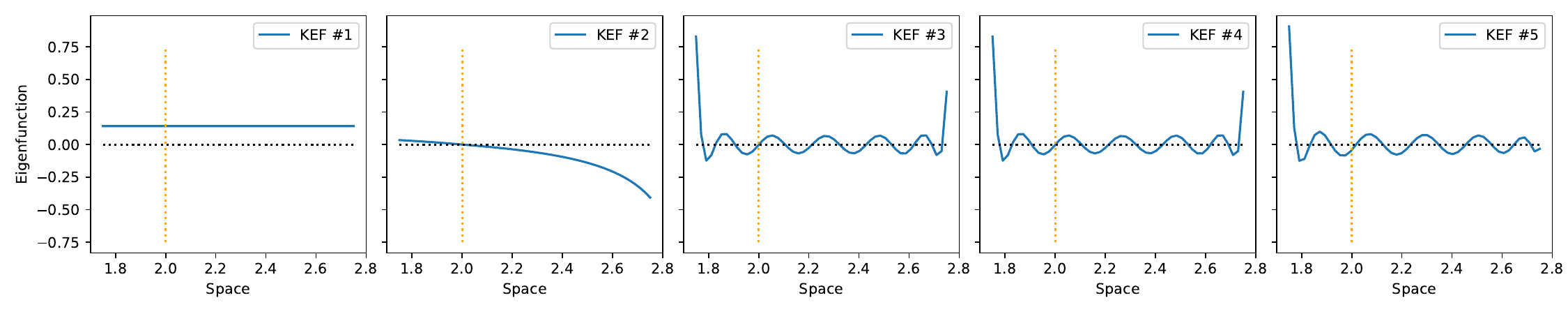}
\includegraphics[width=1\textwidth]{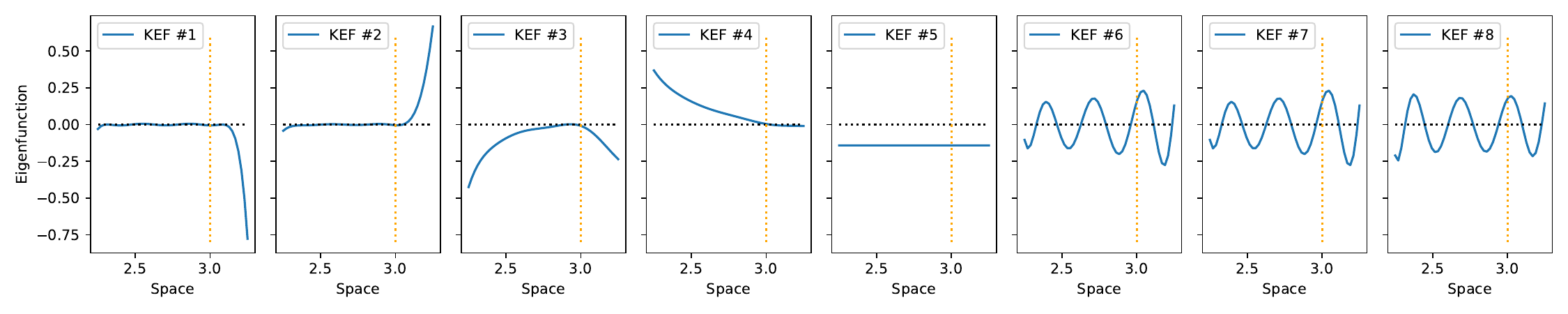}
\caption{\label{fig:LogAbsKEF_1Dapprox_KEFleft}Eigenfunctions approximated with EDMD (Gaussian radial basis function dictionary, kernel bandwidth $\epsilon=0.05$ and $\epsilon=0.15$ for top and bottom rows). Eigenfunctions associated to eigenvalues with nonzero imaginary part are excluded, and the eigenfunctions are normalized to have norm 1. The steady state is denoted by an orange vertical dashed line, indicating that the top row is computed around steady state $x=a=2$, the bottom row around steady state $x=b=3$. The last three  eigenfunctions in each row are spurious and incorrectly oscillate around zero.}
\end{figure}

In the last step, we try to find a linear map from one set of the eigenfunctions (associated with the left steady state, for example) to the other set (here, to the right one) {\em in logarithmic space}. This can only practically be identified (again, due to dictionary limitations) in the intermediate region of the interval between the steady states; we choose the domain $[2.25, 2.75]$. Fig.~\ref{fig:LogAbsKEF_1Dapprox_KEFlog_mapped_intermediate} (top) shows the logarithm of the absolute value of several non-spurious eigenfunctions in this region.
We then solve the following linear systems with a least squares method (incl. Tikhonov regularization) for the coefficients $\vect{c}=(c_1,c_2)$:
\begin{eqnarray*}
\log|\phi_{k_1}| c_1 &=& \log|\phi_{k_2}|,\\
\log|\phi_{k_2}| c_2 &=& \log|\phi_{k_1}|.
\end{eqnarray*}
The coefficients $c_1$ and $c_2$ allow us to map from one set of (logarithms of) eigenfunctions to the other. The result is shown in Fig.~\ref{fig:LogAbsKEF_1Dapprox_KEFlog_mapped_intermediate}.
\begin{figure}[ht!]
\centering
\includegraphics[scale=.38]{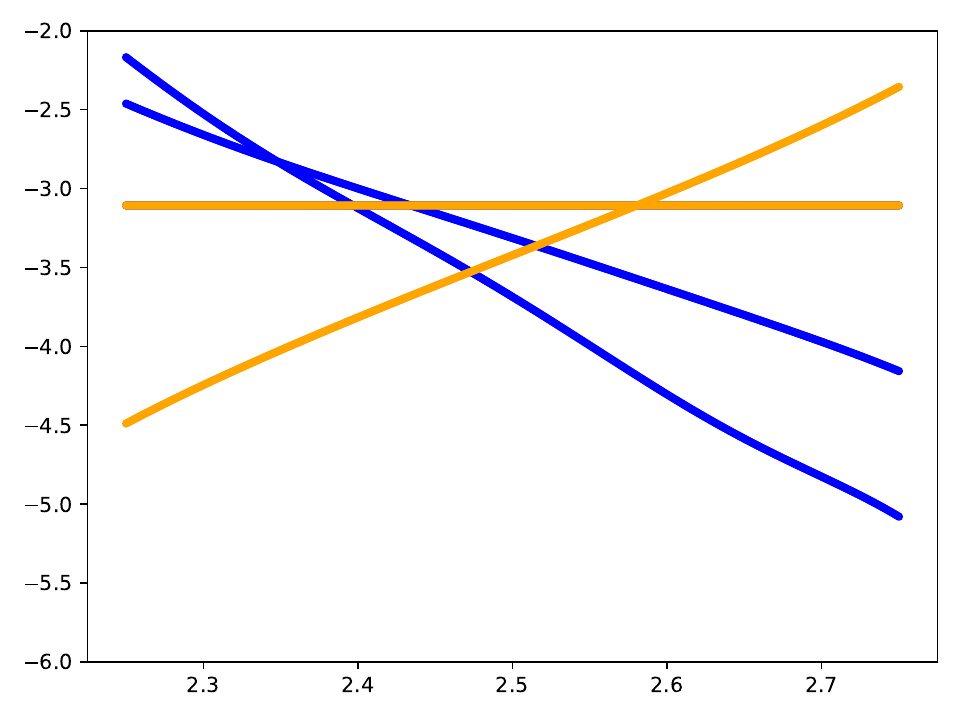}
\includegraphics[width=.9\textwidth]{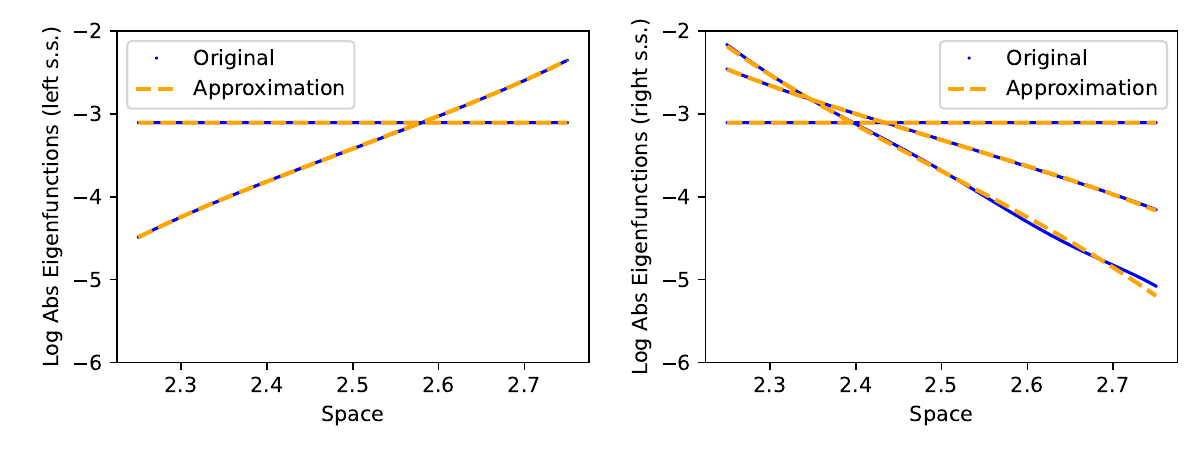}
\caption{\label{fig:LogAbsKEF_1Dapprox_KEFlog_mapped_intermediate}
Top: Logarithms of the absolute value of selected, non-spurious eigenfunctions in the intermediate region $[2.25, 2.75]$ between the steady states. The eigenfunctions associated to $a=2$ are orange, the ones associated to $b=3$ are blue.
Bottom: A linear map from the one set of eigenfunctions (associated to left/right steady state) to the other (right/left) provides a good approximation, meaning we can accurately construct the value of eigenfunctions ``on the other side of the steady state'' (not shown).}
\end{figure}

\subsection{Singularities in two dimensions and the use of isochrons}
We saw that in one dimension, Koopman eigenfunction singularities arise naturally as we approach adjacent steady states (which also naturally constitute the boundary of basins of attraction). 
We now proceed to explore singularities of Koopman eigenfunctions for continuous time dynamical systems {\em in two dimensions}. 
Here, the singularities will arise along entire one-dimensional curves: (a) separatrices, consisting of the stable manifolds of saddle steady states, or (b) limit cycles, surrounding a steady state in the neighborhood of which the Koopman eigenfunction computation is initiated.
Such a limit cycle is also the boundary of the basin of attraction of the steady state (e.g., forward in time if the steady state inside is stable, and the limit cycle unstable; or backward in time in the opposite case). 
In both cases, the important notion is that of an \textit{isochron}, first defined in terms of limit cycles~\cite{mauroy-2013}, but straightforward to also extend in the case of saddle-associated separatrices.
\begin{definition}(Dominant Koopman eigenvalue and eigenfunction)
Consider a dynamical system $ \dot{\vx}=F(\vx)$, where $ \vx \in \mathcal{M}$, and $ F \in C^2 $ in the neighborhood of an asymptotically stable equilibrium $\vx^*$. That is, the eigenvalues $\lambda_j$ of the Jacobian matrix $J(\vx^*)$ satisfy $\text{Re}(\lambda_j) < 0 \quad \text{for all } j$.
In this case, the eigenvalues $\lambda_j$ are directly related to the eigenvalues $\exp(t\lambda_j)$ of the associated Koopman operator ${K}^t_{F^t}$. 

The dominant eigenvalue \(\lambda_1\) is defined as the eigenvalue satisfying
\begin{align}
     \text{Re}(\lambda_1) > \text{Re}(\lambda_j) \quad \text{for all } \lambda_1 \neq \lambda_j.
\end{align}
We assume the existence of such an eigenvalue, which is guaranteed in monotone systems~\cite[Chapter 5.5]{hirsch-2006}.
The eigenfunction associated with $\lambda_1$ is referred to as the dominant eigenfunction and is denoted by $\varphi_{\lambda_1}$.

Under these assumptions, the dominant eigenfunction $\varphi_{\lambda_1}$ can be computed using the Laplace average:
\begin{align}
    f^*_{\lambda_1} (\vx) = \lim_{T \to \infty} \frac{1}{T} \int_0^T \big( f \circ F^t(x) \big) e^{-\lambda_1 t} dt,
\end{align}
    where $f \in C^1$  satisfies $f(\vx^*) = 0$ and $(\nabla f(\vx^*))^T v_1 \neq 0$,
    with  $v_1$  being the right eigenvector of the system Jacobian $J(\vx^*) $ corresponding to $\lambda_1$. The function  $f^*_{\lambda_1} $ represents the dominant eigenfunction $\varphi_{\lambda_1}(\vx)$ up to a scalar multiple \cite{mauroy2020koopman}. 

    For a system with a stable limit cycle $\Gamma$, the dominant Koopman eigenfunction associated with the linearization {\em around the limit cycle} is linked with the leading non-trivial Floquet exponent $\lambda_1$, which governs the slowest transverse decay of trajectories towards the cycle. In an action-angle coordinate system, the function $f_{\lambda_1}^*$ provides a global action coordinate that captures the attractivity of the limit cycle.
\end{definition}


For a dynamical system in $d$ dimensions with an exponentially stable limit cycle $\Gamma$ of period $\omega \in \mathbb{R}^+$, there exist Koopman eigenfunctions~\cite[Chapter 15]{mauroy2020koopman} associated with eigenvalues $\lambda_i$, $i=1,\dots,d$, such that $\lambda_1=\omega$, and $\lambda_i$ are the $(d-1)$ Floquet exponents.
\begin{definition}(Isostables, \cite[Chapter 15]{mauroy2020koopman})\label{def:isostables}
The {\em isostables} of the system are level sets of the Koopman eigenfunctions associated to the Floquet exponents.
Formally, for any value $r\in\mathbb{R}$ and any point $\vx$ in the state space $\mathcal{M}$, an isostable associated to one of the eigenvalues $\lambda_i$ is the set
\begin{align}
    \mathcal{I}_r = \{ \vx \in \mathcal{M} : |\varphi_{\lambda_i}(\vx)| = e^{r \lambda_i} \},
\end{align}
where $\varphi_{\lambda_i}(\vx)$ is the principal Koopman eigenfunction associated with the eigenvalue $\lambda_i$.

For an equilibrium, isostables describe trajectories that decay at the same exponential rate towards $\vx^*$. For a limit cycle, isostables represent sets of points that converge to the cycle with the same phase-independent transient behavior. Isostables provide a global partition of the state space and are closely linked to the stable and unstable manifolds of the system invariant sets, offering valuable insights into the transient dynamics of the system. They can be computed throughout the entire basin of attraction of an attractor using methods such as Laplace averages~\cite{mauroy-2013}.
\end{definition}
\begin{definition}(Isochrons)
    The isochrons of the limit cycle are the level sets of the Koopman eigenfunction $\varphi_{i\omega }$, formally defined as:
    \begin{align}
   \mathcal{I}_\theta = \{ \vx \in \mathcal{M} \mid \angle \varphi_{i\omega}(\vx) = \theta \}, \quad \theta \in [0,2 \pi ).     
    \end{align}
    Each isochron represents a set of initial conditions that asymptotically converge to the same phase on the limit cycle. The evolution of these sets follows the periodic relation $(F^t)^k(\mathcal{I}_\theta) = \mathcal{I}_{\theta + \omega t \, \, (\text{mod} \, 2\pi)}$. In particular, after a full period, the isochrons satisfy $(F^t)^k(\mathcal{I}_\theta) = \mathcal{I}_\theta$,  where $k = \frac{2\pi}{\omega}$ \cite{mauroy2020koopman}.

    Isochrons provide a global phase coordinate system for the dynamics, ensuring that all points on a given isochron approach the same asymptotic phase on the attractor.
\end{definition}

\subsubsection{Isochrons and isostables of limit cycles}\label{sec:example - isochrons of limit cycles}
Consider the Van der Pol system with the parameter $\mu=0.3$. For this parameter value, the system exhibits a stable limit cycle, shown in red in \Cref{fig:isochrons and isostables of limit cycles}, surrounding an unstable (source, focus) steady state.
This figure also illustrates the isochrons (left) and isostables (right) for this Van der Pol system.
Isochrons and isostables are (separately) Koopman eigenfunctions for the system; 
here, they are both anchored on the linearization around the stable limit cycle and their eigenvalues embody the attractivity of the limit cycle as well as its period.
Clearly, the isostable component goes to zero on the limit cycle and to infinity at the steady state in the middle (see \Cref{fig:LC:3d with bottom contour}). 
A negative power of the isostable would then go to infinity on the limit cycle. 

If we also computed the isochrons and isostables associated with the unstable spiral steady state inside the limit cycle, we would encounter a situation similar to the one in the interval between different steady states (different invariant objects) in one dimension. There, we used  the ratios of respective eigenvalues to transform between eigenfunctions. Doing this in the context of two dimensional systems with limit cycles is the context of current research (see~\cite{langfield-2015} for possible structures for ``inner'' and ``outer'' isochrons).

Inside the limit cycle, trajectories converge to it and cannot escape or cross it. Therefore, it is important to examine the region interior, and the one exterior of the limit cycle, along with how these two areas are connected.
%
%
%
%
One effective method for approximating them in the entire domain is to use Laplace averaging~\cite{mauroy-2013} of the observable $g(x_1,x_2)=\sin(x_1+x_2)$.
\begin{figure}[!htbp]
    \centering
    \includegraphics[scale=.5]{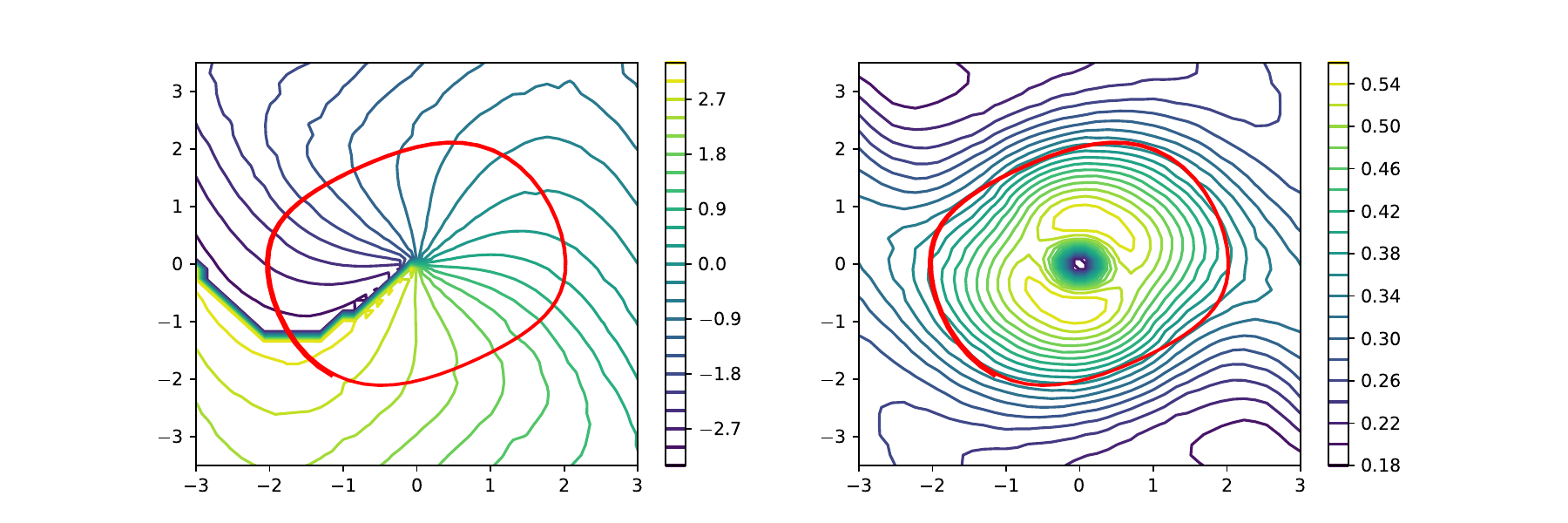}
    \caption{Isochrons (left) and isostables (right) associated with the limit cycle of the Van der Pol system ($\mu=0.3$), computed with Laplace averaging of the observable $g(x_1,x_2)=\sin(x_1+x_2)$.}
    \label{fig:isochrons and isostables of limit cycles}
\end{figure}
%

\subsubsection{Transformation of trajectories across the limit cycle}\label{sec:transformation across limit cycle}
We now describe a procedure to transform trajectories {\em from within the limit cycle} invertibly into trajectories outside of the limit cycle. The key idea is to represent the trajectories in terms of eigenfunctions (isochrons and isostables), and to construct the transformation by transforming between eigenfunctions ``inside'' and eigenfunctions ``outside'' of the limit cycle. 


The procedure works as follows.
\begin{enumerate}
    \item Compute two complex-valued eigenfunctions $\phi_1,\phi_2:\mathbb{R}^2\to\mathbb{C}$ (cf.~\Cref{fig:limitcycle mapping}, leftmost and rightmost plot), where $\angle\phi_1$ and $\angle\phi_2$ are isochrons associated  {\em with  the limit cycle} and eigenfunctions associated {\em with the steady state}, respectively, and $|\phi_1|$, $|\phi_2|$ are the corresponding isostables. The computation of $\phi_1$ is performed in the neighborhood of the limit cycle, the computation of $\phi_2$ in the neighborhood of the steady state.
    \item Extend the functions $\phi_1,\phi_2$ as much as possible (accurately) within the limit cycle, i.e., $\phi_1$ so that its domain approaches the steady state, and $\phi_2$ so that its domain approaches the limit cycle. This is possible by solving the Koopman PDE with the method of characteristics. Note that $|\phi_1|$ approaches infinity towards the steady state (cf. \Cref{fig:isochrons and isostables of limit cycles}, right), while $|\phi_2|$ is zero at the steady state and approaches infinity towards the limit cycle.
    \item Define an invertible transformation function $T_i:\mathbb{C}\to\mathbb{C}$ (subscript $i$ for ``inside'') between the values of $\phi_1$ and the values of $\phi_2$ within the limit cycle. This is possible, because both $\phi_1$ and $\phi_2$ separately can be used to represent the open 2D domain within the limit cycle (excluding it and the steady state).
    \item As isostables (and isochrons) of the limit cycle are defined on either side of it, and both sets have levels from zero to infinity, it is possible to identify isostables inside with isostables outside (cf.~\Cref{fig:limitcycle mapping}, rightmost plot, black circles). Using this identification, it is possible to create a separate transformation $T_o$ (subscript $o$ for ``outside'') that (a) maps isostables ``inside'' the limit cycle to isostables ``outside'' the limit cycle invertibly and (b) does so across the isochrons with identical values inside and outside.
    \item For any trajectory within the limit cycle, it is possible to (a) transform it from Cartesian coordinate representation into an isochron/isostable representation (i.e., range of $\phi_2$), then (b) map this representation to the corresponding one in the range of $\phi_1$ using $T_i$, and finally (c) map this representation to isochrons/isostables outside of the limit cycle using $T_o$. 
\end{enumerate}
\begin{figure}[!htbp]
    \centering
    \includegraphics[width=0.95\textwidth]{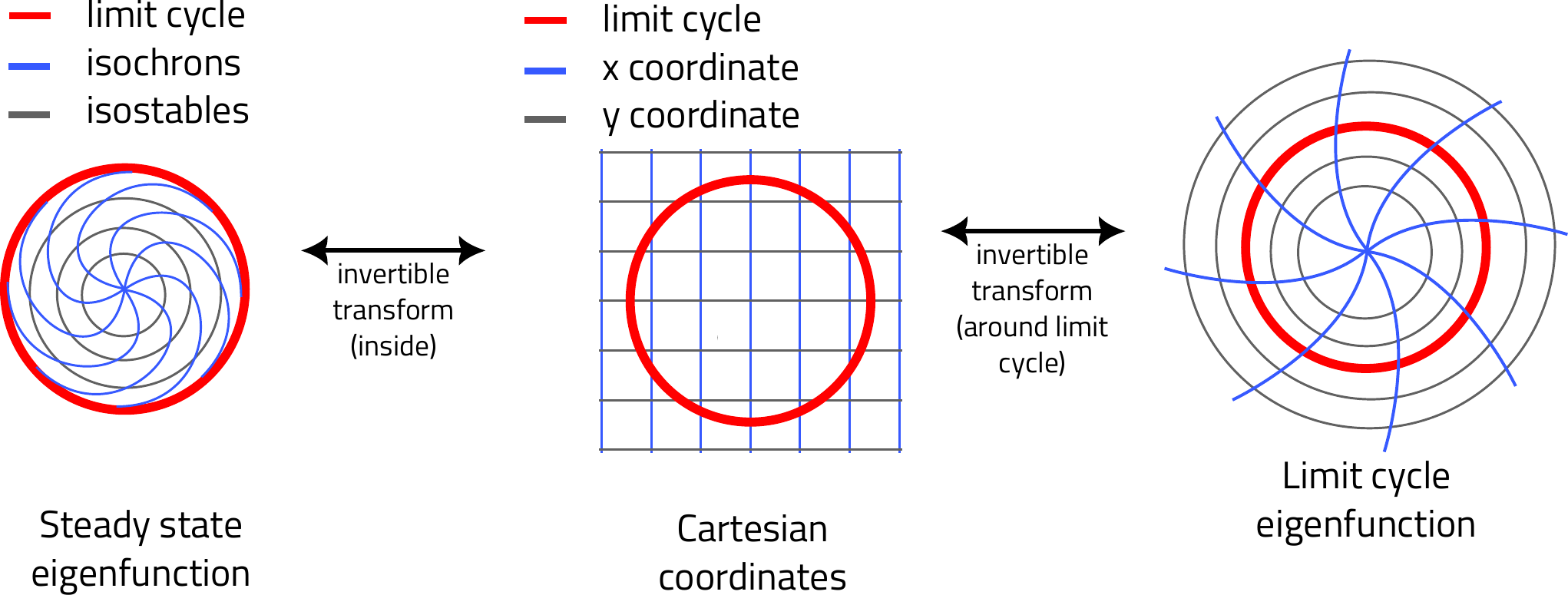}
    \caption{\label{fig:limitcycle mapping}Isochrons and isostables (combined as a single complex-valued Koopman eigenfunction) can be used to construct invertible transforms inside the limit cycle, by transforming to the Cartesian coordinates $(x,y)$. The transformation for the steady state breaks down close to the limit cycle, while the transformation for the limit cycle breaks down close to the steady state.}
    \label{fig:enter-label}
\end{figure}
We now provide explicit formulations for the eigenfunctions and the transformations in a specific example.

{  
\begin{example}
Consider a planar system (in polar coordinates) given by
\begin{align}\label{eq:polar_system}
\begin{cases}
\dot{r} \, = \,  r(\mu - r^2)
\\[1ex]
\dot{\theta} \, = \, \omega + \alpha\,(r - \sqrt{\mu}), \quad \mu,\omega>0, \alpha \in \mathbb{R}   
\end{cases}.
\end{align}
At $ r=0 $, the origin is an unstable equilibrium, and the circle $r = \sqrt{\mu} \,$ is a stable limit cycle (Fig. \ref{fig:LC:SS}).
\begin{figure}[!htbp]
    \centering
\includegraphics[width=0.285\textwidth]{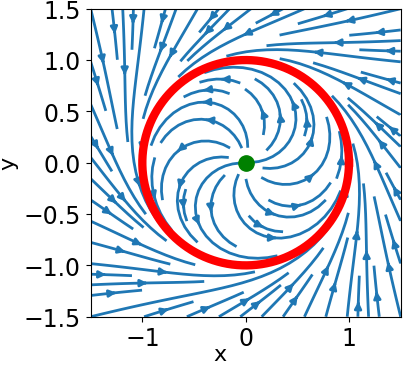}
    \centering
\includegraphics[width=0.35\textwidth]{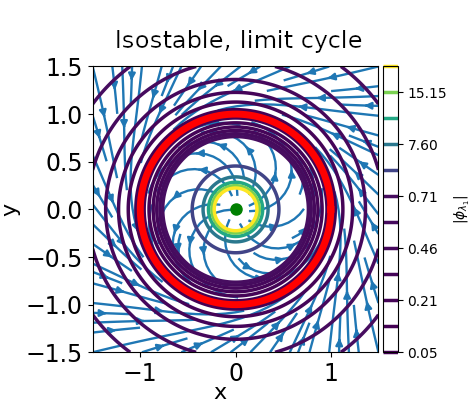}
\includegraphics[width=0.35\textwidth]{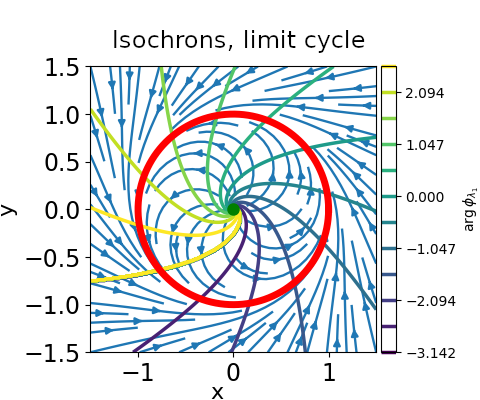}
\includegraphics[width=0.35\textwidth]{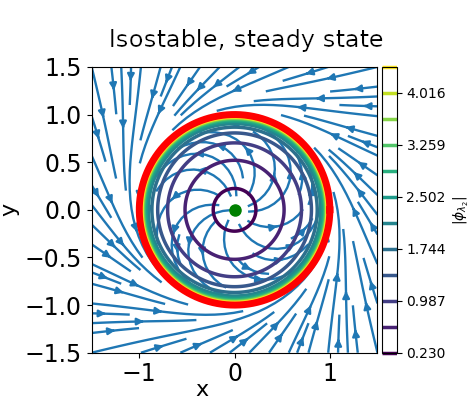}
\includegraphics[width=0.35\textwidth]{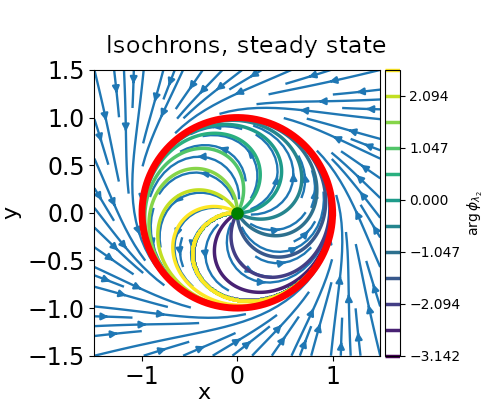}
    \caption{Top: From left to right, the phase portrait of system \eqref{eq:polar_system}, for $\mu = \omega = \alpha = 1$,  with an unstable steady state (green point) and a stable limit cycle (red circle), along with trajectories (light blue). Isostables and isochrons of the limit cycle (top, for $C = 1$) and the steady state (bottom row) are shown in four separate plots. Bottom: The isostables and isochrons of the steady state (for $C = 1$). The way this example is constructed means the linearization of the limit cycle and the steady state are identical, for general dynamical systems with limit cycles this is not the case.}
    \label{fig:LC:SS}
\end{figure}
\begin{figure}[!htbp]
    \centering
\includegraphics[width=1\textwidth]{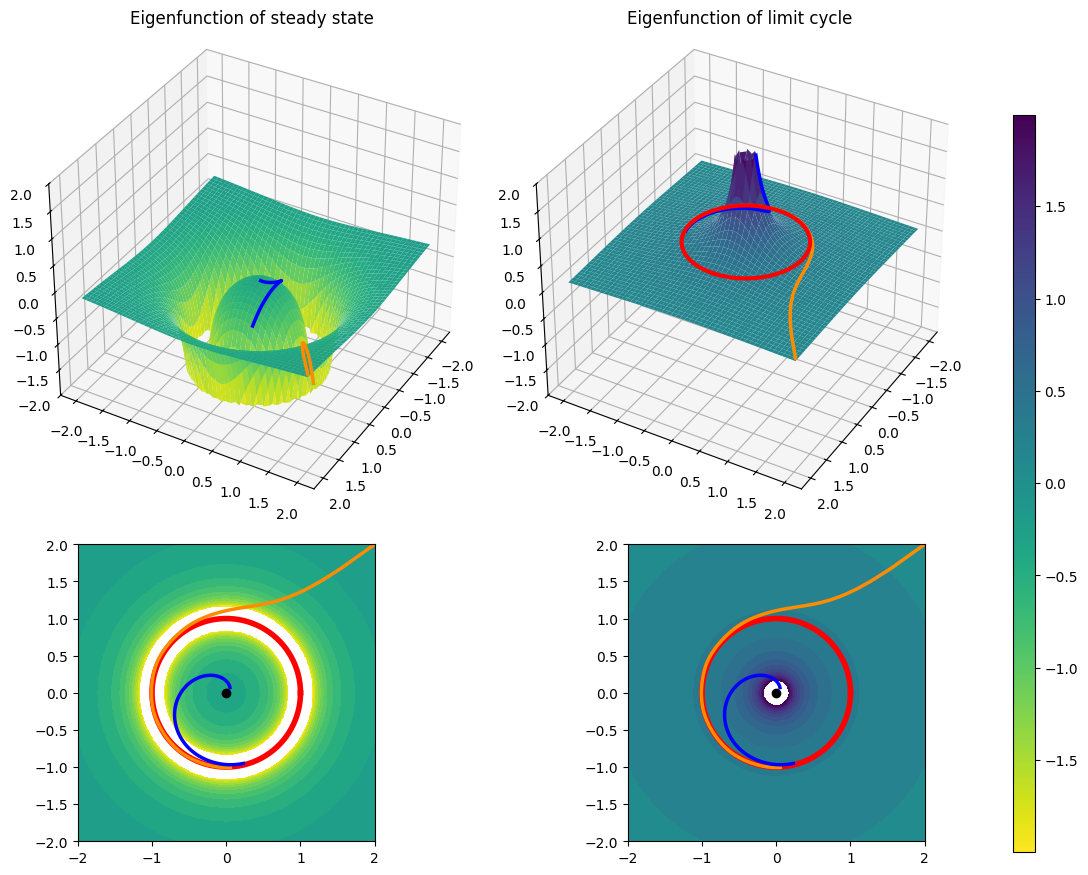}
    \caption{Isostables (Koopman eigenfunctions) for a limit cycle system, plotted in three and two dimensions. Two representative trajectories are shown, with the z-coordinate equal to the respective isostable value (blue inside, orange outside of limit cycle).
    Left: An isostable associated to the steady state. It approaches (negative) infinity on the limit cycle, and is zero on the steady state (also see Fig.~\ref{fig:LC:SS}, bottom left). Outside of the limit cycle, it is extended using the limit cycle isostable.
    Right: An isostable associated to the limit cycle. It is zero on the cycle, and approaches infinity towards the steady state (also see Fig.~\ref{fig:LC:SS}, top, center).}
    \label{fig:LC:3d with bottom contour}
\end{figure}

~\\
\textbf{I. Koopman eigenfunctions :} 
\\[1ex]
The separable eigenfunctions of the form 
\begin{align}
\phi_{\lambda,k}(r,\theta)= h_{\lambda,k}(r)\,e^{ik\theta},  \hspace{.3cm} k \in \mathbb{Z}, 
\end{align}
with eigenvalue $ \lambda \in \mathbb{C}$, can be obtained by solving the Koopman PDE using the method of characteristics: 
\begin{align}
\phi_{\lambda,k}(r,\theta) \,=\, C\;\Bigl(\dfrac{r^2}{|\mu - r^2|}\Bigr)^{\frac{\lambda - ik\omega}{2\mu}}
\;\Bigl(\dfrac{\sqrt{\mu} + r}{r}\Bigr)^{-ik\alpha/ \sqrt{\mu}} \, e^{ik\theta}.   
\end{align}
We consider $C>0$ for the sake of simplicity. Using linearization around the limit cycle $r = \sqrt{\mu } \, $ yields a radial Floquet exponent $ -2\mu$ and angular frequency $\omega$. Thus, for $k=1$ and $\lambda_1 = -2\mu + i\omega$, the eigenfunction associated with the limit cycle becomes
\begin{align}
\phi_{\lambda_1}(r,\theta) \,=\, C\; \dfrac{|\mu - r^2|}{r^2} \cdot \;\Bigl(\dfrac{\sqrt{\mu} + r}{r}\Bigr)^{-i\alpha/ \sqrt{\mu}} \, e^{i\theta}, \hspace{.4cm} C>0.   
\end{align}
Therefore, the isostables of the limit cycle are 
\begin{align}\label{eq:isostables:limit}
\mathcal{I}_{r, \text{LC}} \, = \, | \phi_{\lambda_1}(r,\theta) | \,=\,  C\; \dfrac{|\mu - r^2|}{r^2}, 
\end{align}
and the corresponding isochrons are given by
\begin{align}
\mathcal{I}_{\theta, \text{LC}} \, = \, \angle \phi_{\lambda_1}(r,\theta)  \,=\, \arg(\phi_{\lambda_1}(r,\theta)) \, = \, \theta - \dfrac{\alpha}{\sqrt{\mu}} \, \ln \Bigl(\dfrac{\sqrt{\mu} + r}{r}\Bigr) \hspace{.3cm} (\text{mod} \, \, 2 \pi).
\end{align}
Both isostables and isochrons of the limit cycle are defined {\em on either side of it}. For isostables, we have $\mathcal{I}_{r, \mathrm{LC}} = 0$ on the limit cycle $r = \sqrt{\mu}$. We consider two branches corresponding to the regions inside and outside the limit cycle ($0 < r< \sqrt{\mu} \,$ and $r>\sqrt{\mu}$):
\begin{itemize}
    \item \emph{Interior branch:} $\phi_{\lambda_1}\big|_{(0,\sqrt\mu)\times [0,2 \pi )}$ where $|\mu-r^2|=\mu-r^2>0$. On this punctured open disk  
$$  \phi_{\lambda_1}\big|_{(0,\sqrt\mu)\times [0,2 \pi )}:(0,\sqrt\mu)\times [0,2 \pi )\;\xrightarrow{\;\text{bijective}\;}\; (0, \infty)\times [0,2 \pi ),$$
  and $\mathcal{I}_{r,\mathrm{LC}}(r)= C \big(\frac{\mu}{r^2}-1 \big)\in(0,\infty)$.
    \item \emph{Exterior branch:} $\phi_{\lambda_1}\big|_{(\sqrt\mu,\infty)\times [0,2 \pi )} $ where $|\mu-r^2|=r^2-\mu>0$. On this exterior region
$$
  \phi_{\lambda_1}\big|_{(\sqrt\mu,\infty)\times [0,2 \pi )} :(\sqrt\mu,\infty)\times [0,2 \pi )\;\xrightarrow{\;\text{bijective}\;}\; (0, C)\times [0,2 \pi ),
$$
  and $\mathcal{I}_{r,\mathrm{LC}}(r)=C \big(1 - \frac{\mu}{r^2} \big) \in(0, C)$.
\end{itemize}
Moreover
\begin{itemize}
    \item If $\alpha = 0$, then the isochrons are $\angle \phi_{\lambda_1}(r,\theta)  \,=\, \theta \,$ , which correspond to straight radial lines.
    \item If $\alpha \neq 0$, then the isochrons are curved lines.
\end{itemize}

Similarly, linearization around the unstable steady state $r=0$ gives a radial exponent $\mu$ and angular frequency $\omega - \alpha \sqrt{\mu}$. Hence, for $k=1$ and $\lambda_1 = \mu + i(\omega- \alpha \sqrt{\mu})$, the eigenfunction associated with the steady state is 
\begin{align}
\phi_{\lambda_2}(r,\theta) \,=\, C\; \dfrac{r}{\sqrt{\mu - r^2}} \; \cdot \Bigl(\dfrac{\sqrt{\mu} + r}{\sqrt{\mu - r^2}}\Bigr)^{-i\alpha/ \sqrt{\mu}} \, e^{i\theta}, \hspace{.4cm} 0 \leq r < \sqrt{\mu}, \hspace{.4cm} C>0.   
\end{align}
The isostables of the steady state can be expressed as
\begin{align}\label{eq:isostables:steady}
\mathcal{I}_{r, \text{SS}} \, = \, | \phi_{\lambda_2}(r,\theta) | \,=\,  C\; \dfrac{r}{\sqrt{\mu - r^2}} \in [0, \infty), \hspace{.4cm} 0 \leq r < \sqrt{\mu}, 
\end{align}
and the associated isochrons are 
\begin{align}
 \mathcal{I}_{\theta, \text{SS}} \, = \, \angle \phi_{\lambda_2}(r,\theta)  \,=\, \arg(\phi_{\lambda_2}(r,\theta)) \, = \, \theta - \dfrac{\alpha}{\sqrt{\mu}} \, \ln \Bigl(\dfrac{\sqrt{\mu} + r}{\sqrt{\mu - r^2}}\Bigr) \hspace{.3cm} (\text{mod} \, \, 2 \pi), \hspace{.4cm} 0 \leq r < \sqrt{\mu},
\end{align}
which are straight radial lines for $\alpha = 0$. Furthermore, $\phi_{\lambda_2}$ is defined for $0\le r<\sqrt\mu$ (inside the limit cycle) and satisfies $\phi_{\lambda_2}(0,\theta)=0$. Its restriction to $(0,\sqrt\mu)\times [0,2 \pi )$ is a bijection onto $\mathbb C\setminus\{0\}$, while the full map is \emph{surjective} onto $\mathbb C$ but not injective at $0$ (all $\theta$ map to $0$).\\

From eq. \eqref{eq:isostables:limit} and eq. \eqref{eq:isostables:steady}, we can see that the isostables $| \phi_{\lambda_1}(r,\theta) |$ of the limit cycle are zero at the limit cycle and approach infinity toward the steady state inside (and toward infinity outside). In contrast, the isostables $| \phi_{\lambda_2}(r,\theta) |$ of the steady state are zero at the steady state and approach infinity towards the limit cycle. For more details, see Fig. \ref{fig:LC:SS}. \\[2ex]
\textbf{II. Invertible transformation $T_i$ :}
\\[1ex]
For fixed parameters $\mu, C>0$, and $\alpha\in\mathbb{R}$, the maps $\, \phi_{\lambda_1}\big|_{(0,\sqrt\mu)\times [0,2 \pi )} : (0,\sqrt{\mu}) \times [0,2 \pi )  \longrightarrow  \mathbb{C} \setminus \{0\} \, $ and $\, \phi_{\lambda_2}\big|_{(0,\sqrt\mu)\times [0,2 \pi )} : (0,\sqrt{\mu}) \times [0,2 \pi )  \longrightarrow \mathbb{C} \setminus \{0\} \,$
are bijective on  $(0,\sqrt{\mu})\times [0,2 \pi )$ 
inside the limit cycle on the punctured open disk where the steady state and the limit cycle are excluded. The transformation $T_i$ can be built by composition as 
\begin{align}\nonumber
& T_i :=   \phi_{\lambda_2}\big|_{(0,\sqrt\mu)\times [0,2 \pi )} \circ  \left(\phi_{\lambda_1} \big|_{(0,\sqrt{\mu}) \times [0,2 \pi )}\right)^{-1}: \, \mathbb{C} \setminus \{0\} \; \longrightarrow \; \mathbb{C} \setminus \{0\}  
 \\
& T_i(z) \, = \, \sqrt{\frac{C^3}{|z|}}\exp\!\Big(i\big[\arg z+\tfrac{\alpha}{2\sqrt\mu}\ln \left(\tfrac{|z|}{C} \right)\big]\Big),
\end{align}
which is invertible with the inverse 
\begin{align}\nonumber
& T_i^{-1} = \phi_{\lambda_1}\big|_{(0,\sqrt\mu)\times [0,2 \pi )} \circ  \left(\phi_{\lambda_2} \big|_{(0,\sqrt{\mu}) \times [0,2 \pi )}\right)^{-1}: \, \mathbb{C} \setminus \{0\} \; \longrightarrow \; \mathbb{C} \setminus \{0\}  
 \\
& T_i^{-1}(v)  \, = \, \frac{C^3}{|v|^2}\exp\!\Big(i\big[\arg v - \tfrac{\alpha}{\sqrt\mu}\ln \left(\tfrac{C}{|v|} \right)\big]\Big),
\end{align}
which implies that we can transform the values of $\phi_1$ and the values of $\phi_2$ within the limit cycle.
\\[2ex]
\textbf{III. Invertible transformation $T_o$ :}
\\[1ex]
We construct a transformation $T_o$ that maps the values of interior isostables of the limit cycle to their exterior counterparts while preserving the associated isochron values. 
For the isostables of the limit cycle, we have
\begin{align}
\mathcal{I}_{r, \mathrm{LC}} \,=\, \; 
\begin{cases} 
C \left(\dfrac{\mu}{r^2}-1 \right) \in (0,\infty), & \text{for} \, \, 0 < r < \sqrt{\mu} \; \; (\text{inside}) \\[2ex]
C \left(1- \dfrac{\mu}{r^2}\right) \in (0, C), &  \text{for} \, \, r > \sqrt{\mu} \; \; (\text{outside})
\end{cases}.
\end{align}
If we define a monotone bijection
\begin{align}\nonumber
 \tilde{T}_o :  \;  (0, \infty)  \times & [0,2 \pi ) \;\longrightarrow\; (0, C)  \times [0,2 \pi )  \\[1mm]
& (r, \theta)  \;\longmapsto\; 
\left( \dfrac{C \, r}{1+r}, \; \theta \right),
\end{align}
then the map 
\begin{align}\nonumber
 T_o \; = \; 
 \big(\phi_{\lambda_1}|_{(\sqrt\mu, \infty)\times [0,2 \pi )}\big)^{-1} \circ \; \tilde{T}_o \; \circ \phi_{\lambda_1}\big|_{(0,\sqrt\mu)\times [0,2 \pi )}  :  \;  (0, \sqrt{\mu}) \times [0,2 \pi ) \;\longrightarrow\; (\sqrt{\mu}, \infty)  \times [0,2 \pi ),
\end{align}
defines a bijection between the interior and exterior regions of the limit cycle while preserving the
Koopman phase along the isochrons by construction.
\\[2ex]
\textbf{IV. Mapping a trajectory inside the limit cycle to exterior isostables/isochrons:}
\\[1ex]
Transforming trajectories within the limit cycle to trajectories outside of it is possible if both isostables and isochrons are available. If only isochrons of limit cycle and steady state are available, then CIFT-bifurcations~\cite{langfield-2015} can prevent this type of mapping.
In the former case, any trajectory $\mathbf{x}(t) = (r(t),\theta(t))$ inside the limit cycle, with $0 < r(t) < \sqrt{\mu} \,$ and $t \in I$, can be mapped into an isochron/isostable representation of the steady state using the eigenfunction $\phi_{\lambda_2}$. Next, this representation can be transferred to the isochron/isostable representation {\em associated with the limit cycle} by applying $T_i^{-1}$. Finally, we can map this representation to isostables/isochrons {\em outside of the limit cycle} using $\big(\phi_{\lambda_1}|_{(\sqrt\mu, \infty)\times [0,2 \pi )}\big)^{-1} \circ \; \tilde{T}_o$. 
Thus the overall transformation is
\begin{align}
(r,\theta) 
\;\longmapsto\;  \big(\phi_{\lambda_1}|_{(\sqrt\mu, \infty)\times [0,2 \pi )}\big)^{-1} \circ \; \tilde{T}_o \circ T_i^{-1} \circ \phi_{\lambda_2} \big|_{(0,\sqrt{\mu}) \times [0,2 \pi )}(r,\theta).
\end{align}

\end{example}
}


\subsection{Isochrons for systems with saddle points}\label{sec:example - saddle isochrons}

Figures~\ref{fig_eigenfunctions_ss}--\ref{fig_eigenfunctions_combination_3d} show Koopman eigenfunctions of the bi-stable system (\Cref{system:nonlinear2D}) whose two basins of attraction are separated by the stable manifold of the saddle. Starting with each of the stable steady states, and constructing the Koopman eigenfunctions based on their linearization, we can clearly see these eigenfunctions going to infinity at the separatrix. 
Two important questions arise: is there a way of extending each one of them across the separatrix to the ``other'' basin? Is there Is there (and if yes, can we construct) an analogy with the power transformation in one dimension?
To explore this, we first performed a simple numerical experiment in a linear system that only possesses a single saddle (not multiple basins of attractions, and accordingly no separatrices). 
The results are shown in \Cref{fig:eigenfunction for linear system with saddle}. We visualize that the level sets of the two principal eigenfunctions of the system intersect the stable/unstable manifolds of the saddle exactly in the same way that isochrons and isostables would for a limit cycle.
Here, the limit cycle corresponds to the unstable manifold of the saddle, the level sets of the eigenfunction associated to its dynamics (\Cref{fig:eigenfunction for linear system with saddle}, left) corresponds to the isochrons, and the level sets of the eigenfunction associated to the stable dynamics (\Cref{fig:eigenfunction for linear system with saddle}, left) corresponds to the isostables converging to the limit cycle.
Points on these level sets will, backward in time, approach the unstable eigendirection of the saddle 
and will asymptotically approach backward in time the point at which the level sets cross the unstable eigendirection. In that sense, these level sets are isochrons of the unstable eigendirection. They will help us bridge eigenfunctions defined on alternate sides of it.
For more complicated systems with limit cycles (cf. section~\ref{sec:transformation across limit cycle}) or separatrices (cf. section~\ref{sec:example - saddleSystem}), the same concept can, in principle, be applied.

\subsubsection{Isochrons for a nonlinear systems with a single saddle point}

 Consider a nonlinear 2D system with a saddle at the origin (red dot, with trajectories as gray curves in \Cref{fig:eigenfunction for linear system with saddle}). For this saddle, we can construct one eigenfunction associated with the unstable direction/manifold and another associated with the stable direction/manifold, both of which precisely correspond to the saddle's eigenvalues by construction.

 To construct this example, we start with a linear system associated with the eigenvalues $(\lambda_1,\lambda_2)=(-1,1.5)$ of a saddle point and then transform the state space to obtain a nonlinear system. For the original linear system, the eigenfunctions are analytically available. Here, this linear system is defined by
\begin{eqnarray*}
    \dot{x}_1&=&\lambda_1 x_1,\\
    \dot{x}_2&=&\lambda_2 x_2.
\end{eqnarray*}
Hence, the eigenfunctions are simply $\phi_1(x)=x_1$, $\phi_2(x)=x_2$, and associated to the eigenvalues $\lambda_1,\lambda_2$ of the saddle.
To calculate these eigenfunctions for the saddle in the nonlinear system, a nonlinear transformation (diffeomorphism) is used to transform the state space. Here, we choose $x\mapsto \exp(\frac{1}{\pi} R x) - (1,1)^T$, with $R$ a matrix that rotates $x$ by $60$ degrees and the exponential function applied coordinate-wise. This also transforms the eigenfunctions (through the same function).
%
\Cref{fig:eigenfunction for linear system with saddle}  shows these transformed Koopman eigenfunctions associated with the stable and unstable manifolds of the saddle point of the nonlinear system. 
If we take any points on one contour line associated with the unstable manifold (left plot), they will diverge in the same manner towards infinity in the unstable direction. The right plot illustrates the parameterization of the stable direction. All points on one of the contour lines will converge backwards in time to the same point on the stable manifold.  This is the analog to isochrons being associated with limiting points on the limit cycle, and to isostables being associated with convergence to the cycle. 

\begin{figure}[!htbp]
    \centering
\includegraphics[width=.7\textwidth]{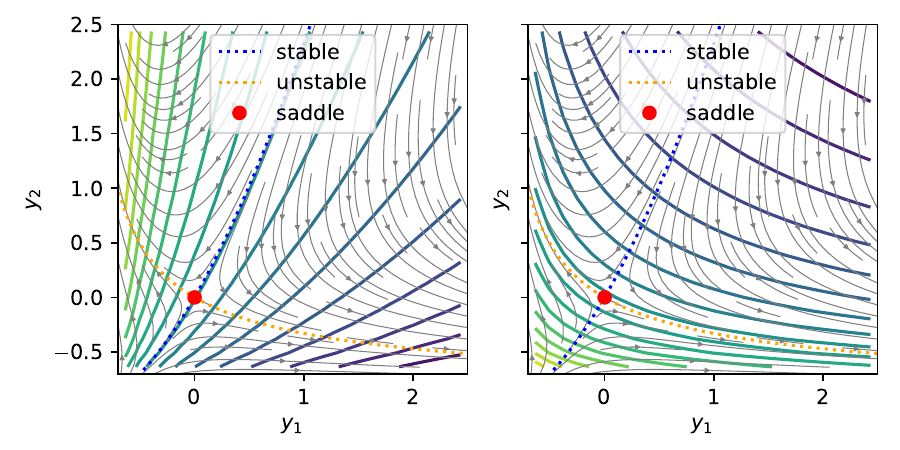}
    \caption{Koopman eigenfunctions (contour lines) of a system with a saddle (red dot) at the origin, and example trajectories as thin, gray curves in the background. The eigenfunctions  
    associate level sets of the state space with points on the unstable (left) and stable (right) manifold of the saddle. }
    \label{fig:eigenfunction for linear system with saddle}
\end{figure}

 \FloatBarrier

\subsubsection{
Separatrices for multistable systems}\label{sec:example - saddleSystem}

%
Consider now a 2D nonlinear system (c.f.~\Cref{fig:2d saddle point nonlinear}, left) given by 
\begin{align}\label{system:nonlinear2D}
\begin{cases}
\dot{y}_1 \, = \, - (y_1 - 1/4)(y_1 + 1/4) \, y_1, \\
\dot{y}_2  \, = \, - y_2
\end{cases}.
\end{align}
This system has a single saddle point at $P_0 = (0, 0)$ 
with eigenvalues $\lambda_{0,1} = 1/16$ and $\lambda_{0,2} = -1$, and two stable nodes at $P_1 = ( -1/4, 0 )$ and $P_2 = ( 1/4, 0 )$, both having eigenvalues
%
 $\lambda_{1,1} = \lambda_{2,1} = -1/8$ and $\lambda_{1,2} = \lambda_{2,2} = -1$. 
To make the discussion more visually compelling, we transform this system using an analytic, non-linear diffeomorphism 
\begin{align}
   \begin{pmatrix}
   x_1 \\ x_2    
   \end{pmatrix} 
=  \vx = h (\vy) = 2\left(\begin{matrix}
    y_1+y_1^4 +2 y_1^2y_2+y_2^2\\y_1^2+y_2
\end{matrix}\right).
\end{align}
The inverse of $h$ is also analytic, and given by
\begin{align}
    \begin{pmatrix}
   y_1 \\ y_2    
   \end{pmatrix} 
= \vy =h^{-1}(\vx)=\left(
\begin{matrix}
    x_1/2-(x_2/2)^2\\
    -(x_1/2)^2+x_2/2 +1/4 (x_1 x_2^2)-(x_2/2)^4
\end{matrix}
\right).
\end{align}

The transformed system has a saddle at $\Tilde{P}_0 = (0, 0)$, and two stable nodes at $\Tilde{P}_1 = ( 0.5078125,  1/8 )$ and $\Tilde{P}_2 = ( -0.4921875,  1/8 )$.
Nearly all initial conditions - except those on the saddle point's stable manifold - will converge towards one of the stable equilibria. The stable manifold of the saddle point is represented by a green curve, and it constitutes the boundary separating the basins of attraction of the two stable nodes. Meanwhile, the two sides of the unstable manifold of the saddle are shown in red and asymptote each to one of the two stable nodes (Fig. \ref{fig:2d saddle point nonlinear}).
\begin{figure}[!htbp]
    \centering
\includegraphics[width=1\textwidth]{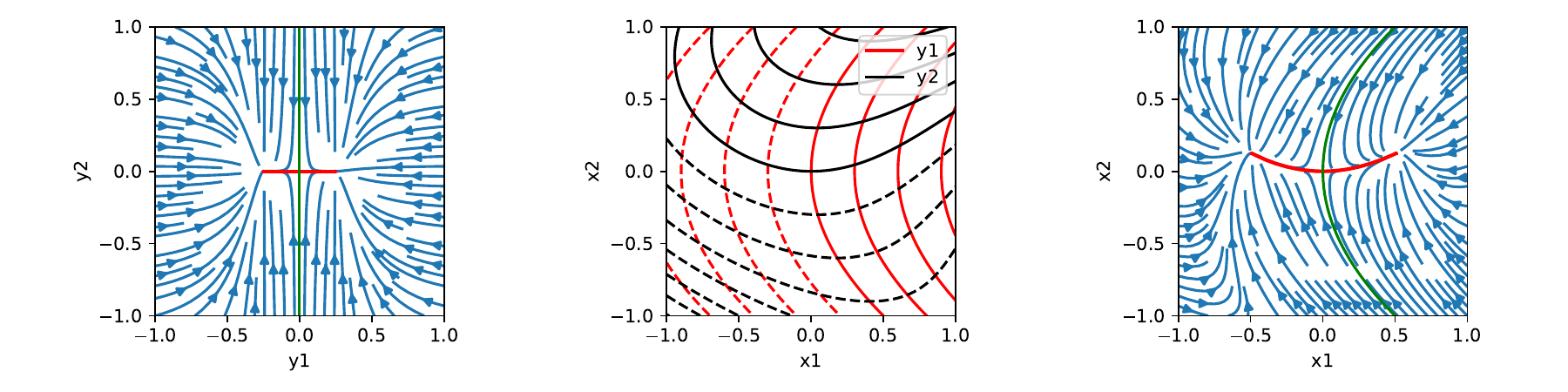}
    \caption{Phase portrait of the 2D nonlinear system \eqref{system:nonlinear2D} with two stable nodes and one saddle. The stable and unstable manifolds of the saddle point are depicted as green and red curves respectively. The left plot shows the system in original coordinates $\vy$, the center plot shows the nonlinear transformation $h^{-1}$ from $\vx$ to $\vy$, and the right plot shows the transformed system in $\vx$ coordinates.}
    \label{fig:2d saddle point nonlinear}
\end{figure}
%
\begin{figure}[!htbp]
\centering
\includegraphics[width=.31\textwidth]{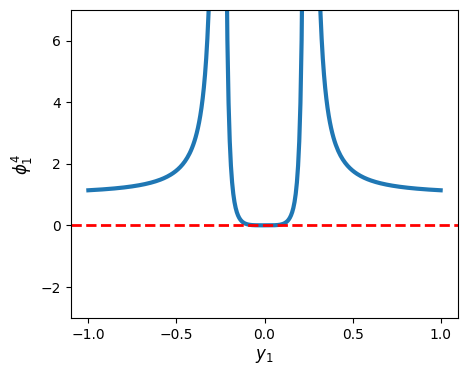}
\hspace{.5cm}
\includegraphics[width=.31\textwidth]{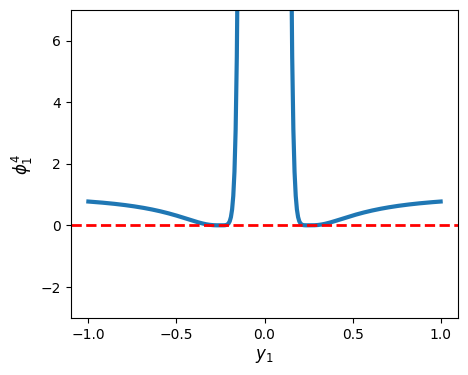}
\hspace{.5cm}
\includegraphics[width=.31\textwidth]{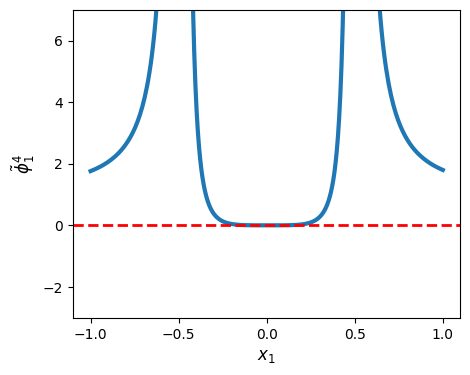}
\hspace{.5cm}
\includegraphics[width=.31\textwidth]{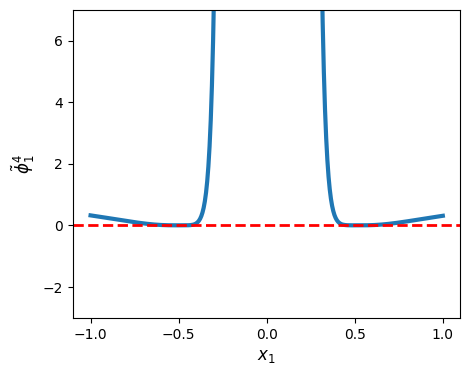}
\caption{Koopman eigenfunction $\phi_1^4$ of the system~\eqref{system:nonlinear2D} in the coordinates $\vy=(y_1,y_2)$ for 
$\lambda_1 = 1/16$ (top left); and 
$\lambda_3 = -1/8$ (top right). 
The transformed eigenfunction $\tilde{\phi}_1^4$ is represented in the coordinates $\vx=(x_1,x_2)$ for 
$\lambda_1 = 1/16$  (bottom left); and 
$\lambda_3 = -1/8$ (bottom right).}\label{fig_eigenfunctions_ss}
\end{figure}
\begin{figure}[!htbp]
\centering
\includegraphics[width=.46\textwidth]{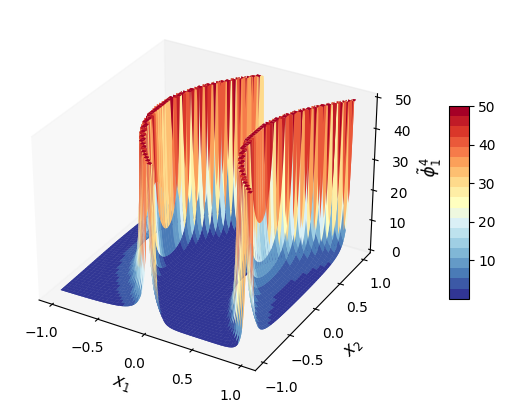}
\includegraphics[width=.46\textwidth]{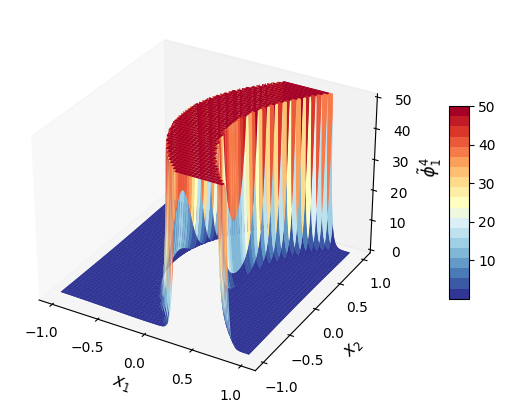}
\caption{The transformed eigenfunction $\tilde{\phi}_1^4$ in 3D space with coordinates $(x_1,x_2, \tilde{\phi}_1^4)$ for 
$\lambda_1 = 1/16$ (left); and 
$\lambda_3 = -1/8$ (right). Note that the functions extend to positive infinity, but the plot shows a flat plateau at $\phi=50$ (where our plotting saturates) instead.}\label{fig_eigenfunctions_ss_3d}
\end{figure}
\begin{figure}[!htbp]
\centering
\includegraphics[width=.32\textwidth]{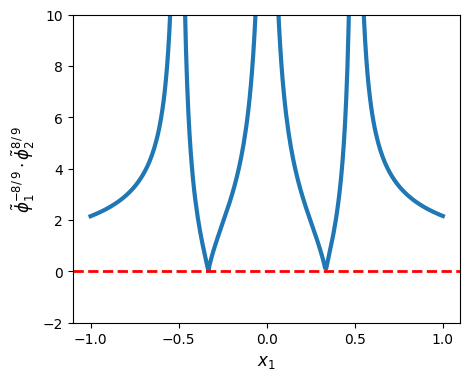}
\hspace{.5cm}
\includegraphics[width=.32\textwidth]{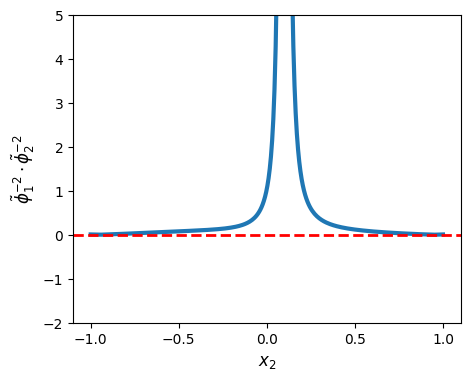}
\caption{Combinations of the transformed eigenfunctions $\tilde{\phi}^{m_1,m_2}_{12} $ in the coordinates $\vx=(x_1,x_2)$ for $\, m_1 = -8/9,  m_2 = 8/9 $ and 
$\lambda_3 = -1/8$ (left); and $\, m_1 =  m_2 = -2 $ and 
$\lambda_1 = 1/16$ (right).}\label{fig_eigenfunctions_combination}
\end{figure}
\begin{figure}[!htbp]
\centering
\includegraphics[width=.46\textwidth]{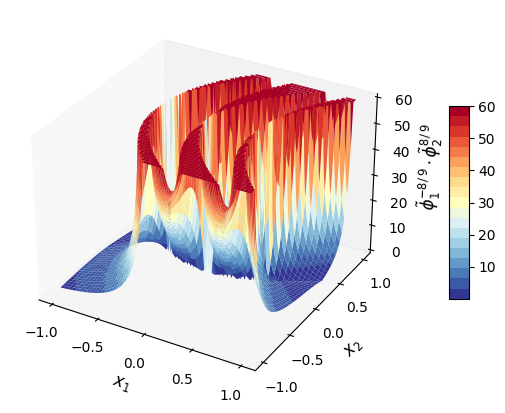}
\includegraphics[width=.46\textwidth]{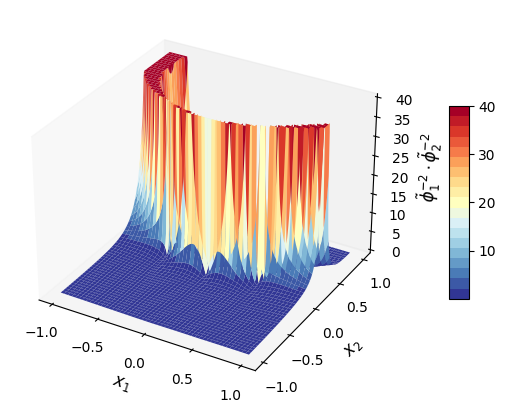}
\caption{Combinations of the transformed eigenfunctions $\tilde{\phi}^{m_1,m_2}_{12} $ in the coordinates in 3D space with coordinates $(x_1,x_2, \tilde{\phi}^{m_1,m_2}_{12})$ for $\, m_1 = -8/9,  m_2 = 8/9 $ and 
$\lambda_3 = -1/8$ (left); and $\, m_1 =  m_2 = -2 $ and 
$\lambda_1 = 1/16$ (right). }\label{fig_eigenfunctions_combination_3d}
\end{figure}
The system in $(x_1,x_2)$ coordinates is constructed through a diffeomorphism $h$ from a simpler system in $(y_1,y_2)$ coordinates, whose Koopman eigenfunctions can be constructed explicitly (and thus also transformed explicitly). 
Although the equilibrium points $P_0, P_1$ and  $P_2$ give rise to six local eigenvalues, there are only three distinct eigenvalues $\lambda_1 = 1/16, \lambda_2 = -1, \lambda_3 = -1/8 $, since some of them are equal.
For system~\eqref{system:nonlinear2D} in coordinates $\vy=(y_1,y_2)$, 
Koopman eigenfunctions $\phi_i$, associated with the distinct eigenvalues $\lambda_i, i= 1, 2, 3$, can be constructed separately for each coordinate, so that   
\begin{align}
\phi^k_{1}(\vy) &\, = \, \Big|(y_1 - 1/4)(y_1 + 1/4) \, y_1^{-2} \Big|^{ k \, \lambda_i/ \lambda_3}, \, \, i=1,3, \hspace{.3cm} k\in\mathbb{Z},
\\[1ex]\nonumber
\text{and}
\\[1ex]
   \phi_2^k(\vy) &\, = \, \big|y_2 \big|^{- k \, \lambda_2} = \,  \big|y_2\big|^k, \hspace{.3cm} k\in\mathbb{Z}.
\end{align}
Of course, by Proposition~\ref{prop: multiplicative_prop}, all combinations $\phi_{12}^{m_1,m_2}=\phi_1^{m_1} \cdot \phi_2^{m_2}$ with $m_1,m_2\in\mathbb{R}$ are also eigenfunctions. \\

The eigenfunctions of the transformed system in $\vx =(x_1,x_2)$ are given by 
\begin{align}
    \tilde{\phi}^{k}_1(\vx) &\, =  \, \phi_1\circ h^{-1}(\vx) 
    \\ \nonumber  
  &\, =  \, \Big|\big(x_1/2-(x_2/2)^2 - 1/4 \big) \big(x_1/2 -(x_2/2 )^2 + 1/4\big) \, \big(x_1/2-(x_2/2)^2 \big)^{-2} \Big|^{k \, \lambda_i/ \lambda_3},
  \\ \nonumber  
  &  
  \, \, \, i=1,3, \hspace{.3cm} k\in\mathbb{Z},
    \\[1ex]\nonumber
& \text{and}
    \\[1ex]
  \tilde{\phi}^{k}_2(\vx) &\, = \, \phi_2\circ h^{-1}(\vx) = \Big| -(x_1/2)^2+x_2/2+ 1/4 (x_1 x_2^2)-(x_2/2)^4 \Big|^k, \hspace{.3cm} k\in\mathbb{Z}.
\end{align}
Similarly, all combinations $\tilde{\phi}^{m_1,m_2}_{12}=  \tilde{\phi}^{m_1}_1 \cdot \tilde{\phi}^{m_2}_2 \, $ for $\, m_1,m_2\in \mathbb{R}$ are again eigenfunctions for the transformed system. See the eigenfunctions plotted in Figures~\ref{fig_eigenfunctions_ss}--\ref{fig_eigenfunctions_combination_3d} and Figure~\ref{fig_eigenfunctions_saddle_four} in Appendix \ref{appendix:fig_bistable}.
%

%
\FloatBarrier
\section{Discussion} \label{sec:discussion}
%
We introduce an approach to numerically construct a large family of eigenfunctions of the Koopman operator, given a small number of eigenvectors from a Koopman matrix, approximated using extended dynamic mode decomposition.
Because integer powers amplify numerical errors exponentially fast, we also introduce error metrics to track how many new eigenfunctions we can credibly (i.e. accurately enough) construct.
We demonstrate the approach in several computational experiments.

The construction of a larger set of eigenfunctions has several benefits. Since integer powers are typically not linearly related to each other, the newly constructed eigenfunctions help to span a larger subspace of functions over the data. This allows us to approximate general observables of dynamical systems more accurately.
Another important observation concerns the construction of eigenfunctions across singularities. We demonstrate how such eigenfunctions arise and how they approach infinity as the input approaches the singularity.
We then show how numerical approximations of the functions on both sides of singularities allow us to construct numerical approximations of the entire function, across the singularity.
%
If we used neural networks with singularities to approximate such eigenfunctions, we would need tools such as rational activation functions~(see \cite{derevianko-2025} for a recent pre-print; \cite{boulle2020rational} discuss rational functions, but without singularities).

The representation of trajectories in terms of eigenfunctions poses challenging new questions for systems with bifurcations related to those eigenfunctions (e.g., so-called ``cubic isochron foliation tangency'' bifurcations, CIFT, see~\cite{langfield-2015}). Such qualitative changes of the system are only visible through tangential crossing of isochrons, but not apparent in the system's dynamic behavior. 
The representation of the state in terms of isochron eigenfunctions becomes singular at the bifurcation point, and thus may provide a path to analyze this behavior.

\backmatter

%
\section*{Declarations}

\begin{itemize}
\item \textbf{Funding}:
The authors are grateful to Prof. Alex Townsend for several comments on the manuscript. F.D. was partially funded by the German Research Foundation/DFG, project 468830823. Z.M. is grateful to the Bundesministerium für Forschung, Technologie und Raumfahrt (BMFTR, Federal
Ministry of Research, Technology and Space) for funding through project OIDLITDSM, No. 01IS24061. 
The work of IGK was partially supported by the US Department of Energy and the US National Science Foundation.\\[1ex]
\item \textbf{Conflict of interest/Competing interests}:
The authors declare that they have no Conflict of interest.
\\[1ex]
\item \textbf{Ethics approval and consent to participate}: Not applicable.
 \\[1ex]
\item \textbf{Consent for publication}:
Not applicable. \\[1ex]
\item \textbf{Data availability}: No datasets were generated during the current study.
\\[1ex]
\item \textbf{Materials availability}:
Not applicable.\\[1ex]
\item \textbf{Code availability}: All codes developed in this study will be made publicly available upon publication.
\\[1ex]
\item \textbf{Author contribution}: F.D.; Y.G.K. devised the main ideas, Z.M.; F.D.; Y.G.K. devised the experiments and wrote the paper; S.M. and S.H. conducted experiments and computed the numerical error estimates.
\end{itemize}

 \bibliography{sn-bibliography}

%

\newpage
\appendix
\section*{Appendix}
\section{Ergodic dynamical systems}\label{sec:ergodic system}

Ergodic systems provide a setting in which the Koopman operator is well-studied. For these dynamical systems, the evolution function $T$ is invertible and measure-preserving, i.e., there exists an invariant measure $\mu$ such that for any $S \in \mathfrak{B}$,
$\mu(S) = \mu(\flow^{-1}(S)),$
where $\flow^{-1}(S)$ is understood as the pre-image of the set $S \in \mathfrak{B}$, i.e.,
$\flow^{-1}(S) := \{\vx \in \mathcal{M} \mid \flow(\vx) \in S\}.$
Ergodic theory then describes statistical properties of such dynamical systems and their long-term behavior.
The theory is developed in the context of measure-preserving transformations in measure theory. Ideas related to the Koopman operator are based on dynamical systems theory (work from Budi\v{s}i\'{c}, Mohr, and Mezi\'{c} \cite{applied-koopmanism:2012}), specifically, ergodic theory (work from Cornfeld, Fomin, and Sinai \cite{ergodic-theory:2012}). Ergodic theory has applications in physics, probability theory, and information theory. We now briefly introduce ergodic theory following Cornfeld, Fomin, and Sinai \cite{ergodic-theory:2012}.
A measure-preserving function $g$ is called \textit{invariant} with respect to the evolution function $\flow$ if for all $x\in \mathcal{M},$
$g(\flow(\vx)) = g(\vx) = g(\flow^{-1}(\vx)).$
If this is true almost everywhere instead of for all $\vx \in \mathcal{M}$, namely, if it is true for $\vx \in \mathcal{M}\setminus \{B \in \mathfrak{B} \mid \mu(B) = 0\}$, 
then it is said to be \textit{invariant} mod $0$. 
Analogously, a set $A \in \mathfrak{B}$ is called \textit{invariant}, or \textit{invariant} mod $0$, if the indicator function 
$\chi_A (\vx)$ (with $\chi_A(\vx)=1$ if $\vx \in A$ and $\chi_A(\vx)=0$ otherwise) is an invariant function, or an invariant mod $0$ function, respectively.
A dynamical system $(\mathcal{M}, \mathfrak{B}, \mu, \flow)$ is called \textit{ergodic} if the measure $\mu(A)$ of any invariant set $A$ equals $0$ or $1$. There are several equivalent conditions of ergodicity, and we use the following definition.
\begin{definition}(Equivalent definition of ergodic system)
    A dynamical system $(\mathcal{M}, \mathfrak{B}, \mu, \flow)$ is called \textit{ergodic} if any function $f\in L^2(\mathcal{M}, \mathfrak{B}, \mu)$ that is invariant with respect to the Koopman operator $\koop$ is constant almost everywhere.
\end{definition} 

\section{Algorithms}
\label{apx:algorithms}
\subsection{Eigensolvers}

\begin{algorithm}
\caption{Deflation based algorithm for real non-Hermitian matrix $A$ with real eigenvalues}
\label{alg: deflation_asymmetric}
\begin{algorithmic}[0]
    \State $i \gets 0$.
    \While{$i < n$}
     \State $\lambda_i, \vv_i \gets \text{power iteration}(A)$.
     \State $\lambda_i, \vw_i \gets \text{power iteration}(A^T)$.
     \State $\vw_i \gets \dfrac{\vw_i}{\vw_i^T \vv_i}$.
     \State $A \gets A - \lambda_i \vv_i \vw_i^T$.
     \State $i \gets i+1$.
    \EndWhile
\end{algorithmic}
\end{algorithm}

\begin{algorithm}
\caption{(power iteration complex) Arnoldi-based power iteration algorithm for real non-Hermitian matrix $A$ with complex dominant eigenvalues $\lambda_{max}, \bar{\lambda}_{max}$ and eigenvectors $\vv, \bar{\vv}$, given tolerance and maximum iterations $N$}
\label{alg: power_iteration_asymmetric_complex}
\begin{algorithmic}[0]
    \State $\vx \gets \text{power iteration}(A, \text{max iterations}=500)$.
    \State $\lambda_o \gets \infty$.
    \State $i \gets 0$.
    \While{$i < N$}
     \State $V[:,0] \gets \vx$.
     \State $h,V \gets \text{modified Gram-Schmidt}(A, V, \text{degree}=2)$.
     
     \State $\lambda_1,\lambda_2 \gets \text{eigenvalue2D}(h)$.
     \State $k  \gets \text{arg}\max_{1,2}\{\abs{\lambda_1}, \abs{\lambda_2}\}$.
     \State $\lambda_{max} \gets \lambda_k$
     \State $\vw \gets \text{eigvector2D}(h, \lambda_{max})$.
     
     \State $\vv \gets V[:\, 0:2]\vw $.
     \State $\vv \gets \dfrac{\vv}{\norm{\vv}}$.
    
     \If{$ \abs{\lambda_{max} - \lambda_o} < \text{tolerance}$}
        \State \textbf{break}.
        \Else
        \State $\lambda_o \gets \lambda_{max}$.
    \EndIf
     \State $i \gets i+1$.
    \EndWhile
\end{algorithmic}
\end{algorithm}

\section{Proofs of Propositions} \label{proofs}
\subsection{Proof of Proposition \ref{prop: upper_bound_continuous_prop}}\label{proof-prop: upper_bound_continuous_prop}
\begin{proof}
Using the definition of Koopman operator,
\begin{align*}
      \phi(F^{\Delta t}(\vx)) & = \phi(\vx^{\Delta t} + \varepsilon(x))\\
                            & \approx \phi(\vx^{\Delta t}) + \varepsilon(\vx)^T \nabla \phi(\vx^{\Delta t})\\
                            & = \lambda^{\Delta t} \phi(\vx) + \varepsilon(\vx)^T \nabla\phi(\vx^{\Delta t}).
\end{align*}

Using the bound on the spectral norm on the Jacobian of the dictionary basis $\Psi$, we get
\begin{align}
    \norm{\nabla\phi(\vx)} = \norm{J_{\Psi}(\vx) \vw}\leq \norm{J_{\Psi}(\vx)}_2 \norm{\vw} \leq L, \;\; \forall \vx \in G,
\end{align}
where $\norm{.}_2$ is the spectral norm. Then using the above equations, the Cauchy-Schwarz inequality, the Binomial theorem, and $\phi_p \equiv \phi^p$, we get
\begingroup
\allowdisplaybreaks
\begin{align*}
    E_{FG}(\bar{\phi}_p, \lambar_p)^{2 p} & =
    \norm{ \phibar_{p}(\flow^{\Delta t} (\cdot) ) - \lambar_{p} \phibar_{p}(\cdot)}_G^{2} =
    \norm{ \bar{\phi}(F^{\Delta t}(\cdot) )^p - (\lambar^{\Delta t} \bar{\phi}(\cdot))^p}_G^2 \\ 
    & = \frac{1}{\abs{G}}\sum_{x \in G} \bigg((\bar{\phi}(F^{\Delta t}(\vx)))^p - (\lambar^{\Delta t}\bar{\phi}(\vx)))^p\bigg)^2 \\
    & = \frac{1}{\abs{G}} \sum_{x \in G} \bigg((\lambar^{\Delta t} \bar{\phi}(\vx) + \varepsilon(\vx)^T \nabla \bar{\phi}(x^{\Delta t}))^p - (\lambar^{\Delta t} \bar{\phi}(\vx)))^p\bigg)^2\\
    & = \frac{1}{\abs{G}} \sum_{x \in G} \bigg( \sum_{k=0}^p \binom{p}{k} (\lambar^{\Delta t} \bar{\phi}(\vx))^{p-k} (\varepsilon(\vx)^T \nabla \bar{\phi}(x^{\Delta t}))^k  - (\lambar^{\Delta t}\bar{\phi}(\vx)))^p\bigg)^2\\
    & = \frac{1}{\abs{G}} \sum_{x \in G} \bigg( \sum_{k=1}^p\binom{p}{k}  (\lambar^{\Delta t} \bar{\phi}(\vx))^{p-k} (\varepsilon(\vx)^T \nabla \bar{\phi}(x^{\Delta t}))^k  \bigg)^2\\
    & =  \frac{1}{\abs{G}} \sum_{x \in G} \bigg( \sum_{k=1}^p \binom{p}{k}  (\lambar^{\Delta t} w^T \Psi(\vx))^{p-k} (\varepsilon(\vx)^T \nabla \bar{\phi}(x^{\Delta t}))^k  \bigg)^2\\
    & \leq \frac{1}{\abs{G}} \sum_{x \in G} \bigg( \sum_{k=1}^p \binom{p}{k}  (\lambar^{\Delta t})^{p-k} \norm{w}^{p-k} \norm{\Psi(\vx)}^{p-k} \norm{\varepsilon(\vx)}^k \norm{\nabla \bar{\phi}(x^{\Delta t})}^k  \bigg)^2\\
    & \leq \frac{1}{\abs{G}} \sum_{x \in G} \bigg( \sum_{k=1}^p \binom{p}{k} (\lambar^{\Delta t})^{p-k}  M^{p-k} \norm{\varepsilon(\vx)}^k L^k  \bigg)^2\\
    & \leq \frac{1}{\abs{G}} \sum_{x \in G}  \bigg( \sum_{k=0}^p \binom{p}{k}  (\lambar^{\Delta t})^{p-k}  M^{p-k} \norm{\varepsilon(\vx)}^k L^k  -  (\lambar^{\Delta t})^{p}M^p \bigg)^2\\
    & \leq \frac{1}{\abs{G}} \sum_{x \in G}  \bigg( \big( \lambar^{\Delta t}  M + L\norm{\varepsilon(\vx)} \big)^p -  (\lambar^{\Delta t})^{p}M^p \bigg)^2\\
    & = \norm{\big( \lambar^{\Delta t}  M + L\norm{\varepsilon(\cdot)}\big)^p -  (\lambar^{\Delta t} M)^{p}}_G^2\\
    & \leq \bigg( \big( \lambar^{\Delta t}  M + L\epsilon_G\big)^p -  (\lambar^{\Delta t} M)^{p} \bigg)^2.
\end{align*}
\endgroup
\end{proof}

\subsection{Proof of Proposition \ref{prop: upper_bound_discrete_prop}}\label{proof-prop: upper_bound_discrete_prop}
\begin{proof}
Using the identity $(a^p -b^p)^2 = (a - b)^2 \big(\sum_{i=0}^{p-1} a^{p - 1 - i}b^i\big)^2$, and the Cauchy–Schwarz inequality, we get
\begin{align*}
E_{FG}(\bar{\phi}_p, \lambar_p)^{2 p} & = \norm{ \phibar_p(F (\cdot) ) - \lambar_p\phibar_p(\cdot)}_G^2\\
    & = \norm{ \big(w_c^T \Psi (F (\cdot)\big)^p - \big(\lambar w_c^T \Psi(\cdot)\big)^p}_G^2\\
    & = \dfrac{1}{\abs{G}} \sum_{\vx \in G} \bigg( (w_c^T \Psi(F(\vx)))^p - ( w_c^T \lambar \Psi(\vx))^p  \bigg)^2\\
    & = \dfrac{1}{\abs{G}} \sum_{\vx \in G} \big(w_c^T \Psi(F(\vx)) - w_c^T \lambar \Psi(\vx))^2 \bigg(\sum_{i=0}^{p-1} (w_c^T \Psi(F(\vx)))^{p-1-i} (w_c^T \lambar \Psi(\vx))^i \bigg)^2\\
    & = \dfrac{1}{\abs{G}} \sum_{\vx \in G} (\delta w^T (\Psi(F(\vx)) - \lambar \Psi(\vx))^2 \ \bigg(\sum_{i=0}^{p-1} (w_c^T \Psi(F(\vx)))^{p-1-i} (w_c^T \Psi(\vx))^i \lambar^i \bigg)^2\\
    & \leq \dfrac{1}{\abs{G}} \sum_{\vx \in G} \norm{\delta w}^2 \norm{\Psi(F(\vx)) - \lambar \Psi(\vx)}^2 \bigg(\sum_{i=0}^{p-1} \norm{w_c}^{p-1-i}\norm{\Psi(F(\vx))}^{p-1-i} \norm{w_c}^i \norm{\Psi(\vx)}^i \lambar^i \bigg)^2\\
    & = \dfrac{1}{\abs{G}} \sum_{x \in G} \norm{\delta w}^2 \norm{\Psi(F(\vx)) - \lambar \Psi(\vx)}^2 \bigg(\sum_{i=0}^{p-1} \norm{\Psi(F(\vx))}^{p-1-i}  \norm{\Psi(\vx)}^i \lambar^i \bigg)^2\\
    & = \norm{\delta w}^2 \norm{  \norm{\Psi(F (\cdot)) - \lambar \Psi(\cdot)} \bigg(\sum_{i=0}^{p-1} \norm{\Psi(F (\cdot))}^{p-1-i}  \norm{\Psi(\cdot)}^i \lambar^i \bigg) }_G^2\\
    & =  C_{TG}(p, \lambar)^2 \norm{\delta w}^2.
\end{align*}
\end{proof}
\section{Analysis of real powers of complex numbers}\label{appendix:analysis of real powers}
Recall that the real-exponential of a complex number is defined for $\alpha \in \mathbb{R}$ and $z = re^{i\theta} \in \mathbb{C}, r, \theta \in \mathbb{R}$ as
$$z^\alpha := e^{\alpha \log(z)},$$
where the complex logarithm $\log(z)$ is defined as
$$\log(z) = \ln(r) + i(\theta + 2\pi k), k\in \mathbb{Z}.$$
Note that for $r=1$, $\log(z) = i(\theta + 2\pi k) \equiv \theta \pmod{2\pi}$ but $\log(e^x) \neq x$ in general. Furthermore, consider $e^{ik\alpha x} = e^{ix\pi + 2n\pi}, n\in \mathbb{Z}$. For $(e^{ik x})^\alpha$, take $\theta = \arg{e^{ik x}} = kx + 2n\pi \in [-\pi, \pi]$, and return $e^{i\theta \alpha}$. However, for $\alpha \in \mathbb{R}\setminus \mathbb{Z}$,
    $$e^{i\alpha x\pi} = (e^{(ix\pi + 2n\pi)})^\alpha = e^{i\alpha x\pi + 2\alpha n\pi} = e^{i\alpha x\pi + 2n'\pi}, n' = \alpha n \in \mathbb{R} \setminus \mathbb{Z}.$$
Thus, the angle is not unique $\pmod{2\pi}$.
The following is a small numerical experiment for calculating $e^{ik\alpha}$, where $i$ is the imaginary unit, $k \in \mathbb{Z}$, and $\alpha \in \mathbb{R}\setminus \mathbb{Z}$. As an example, consider $k = 4$ and $\alpha = 0.5$. Define the following three functions in Python:
\begin{itemize}
    \item $e1(x) = e^{ik\alpha x}$
    \item $e2(x) = (e^{i\alpha x})^k$
    \item $e3(x) = (e^{ikx})^\alpha$
\end{itemize}
These functions should represent the same function, however, the function, which does real power later, shows a different function as is shown in the following image.
Furthermore, consider for $f(\vx) = e3(\vx) = (e^{ikx})^\alpha$. Since $e^{ikx}$ is $\frac{2\pi}{k}$ periodic function (with respect to $x$), $f(\vx)$ is also $\frac{2\pi}{k}$ periodic. By definition, $f(\vx) = e^{\alpha \log(e^{ikx})} = e^{i\alpha (kx + 2n\pi)}$ with $n\in \mathbb{Z}, kx + 2n\pi \in [\pi, \pi)$. Thus, consider the value of $f(\vx)$ in one period, for $m\in \mathbb{Z}$,
\begin{align*}
    &\frac{2\pi}{k} m \leq x < \frac{2\pi}{k}(m+1)\\
    &\therefore 2\pi m \leq kx < 2\pi (m+1)\\
    &\therefore 2\pi (m+n) \leq kx + 2\pi n < 2\pi (m+n+1)\\
    &\therefore 2\pi (m+n) \alpha \leq (kx + 2\pi n) \alpha < 2\pi (m+n+1) \alpha\\
    &\therefore e^{2\pi (m+n) \alpha} \leq e^{(kx + 2\pi n) \alpha} = f(\vx) < e^{2\pi (m+n+1) \alpha}.
\end{align*}
Here, for small $\alpha \approx 0$, both $e^{2\pi (m+n) \alpha}$ and $e^{2\pi (m+n+1) \alpha}$ are close to $1$, and thus, the values $f(\vx)$ are close to $1$. Hence, for small $\alpha$, $f(\vx)$ always takes value close to $1$.
%
\section{Additional figures for the non-linear system \eqref{eq: nonlin_transformed_eqn} in Sect. \ref{sec:example - nonlinear system 2d}}\label{appendix:figures_NTLS}

\begin{figure}[htbp]
    \centering
    \includegraphics[width=.91\textwidth]{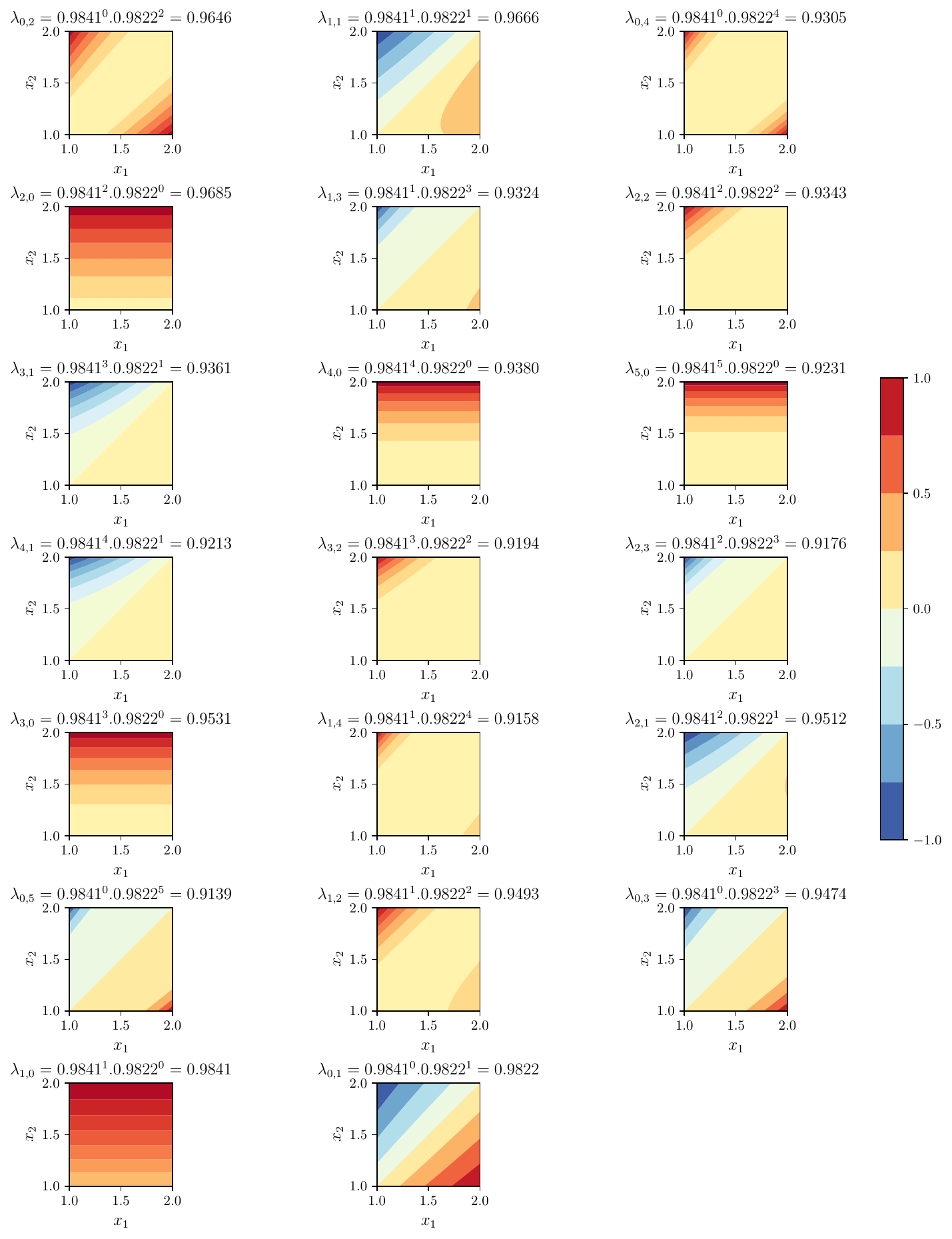}
    \caption{Explicit eigenfunctions for the non-linear system (\ref{eq: nonlin_transformed_eqn}) computed using $\phi \circ h^{-1}$.}
    \label{fig: explicit_eigenfunctions_nonlin_from_lin}
\end{figure}
\begin{figure}[htbp]
    \centering
    \includegraphics[width=.95\textwidth]{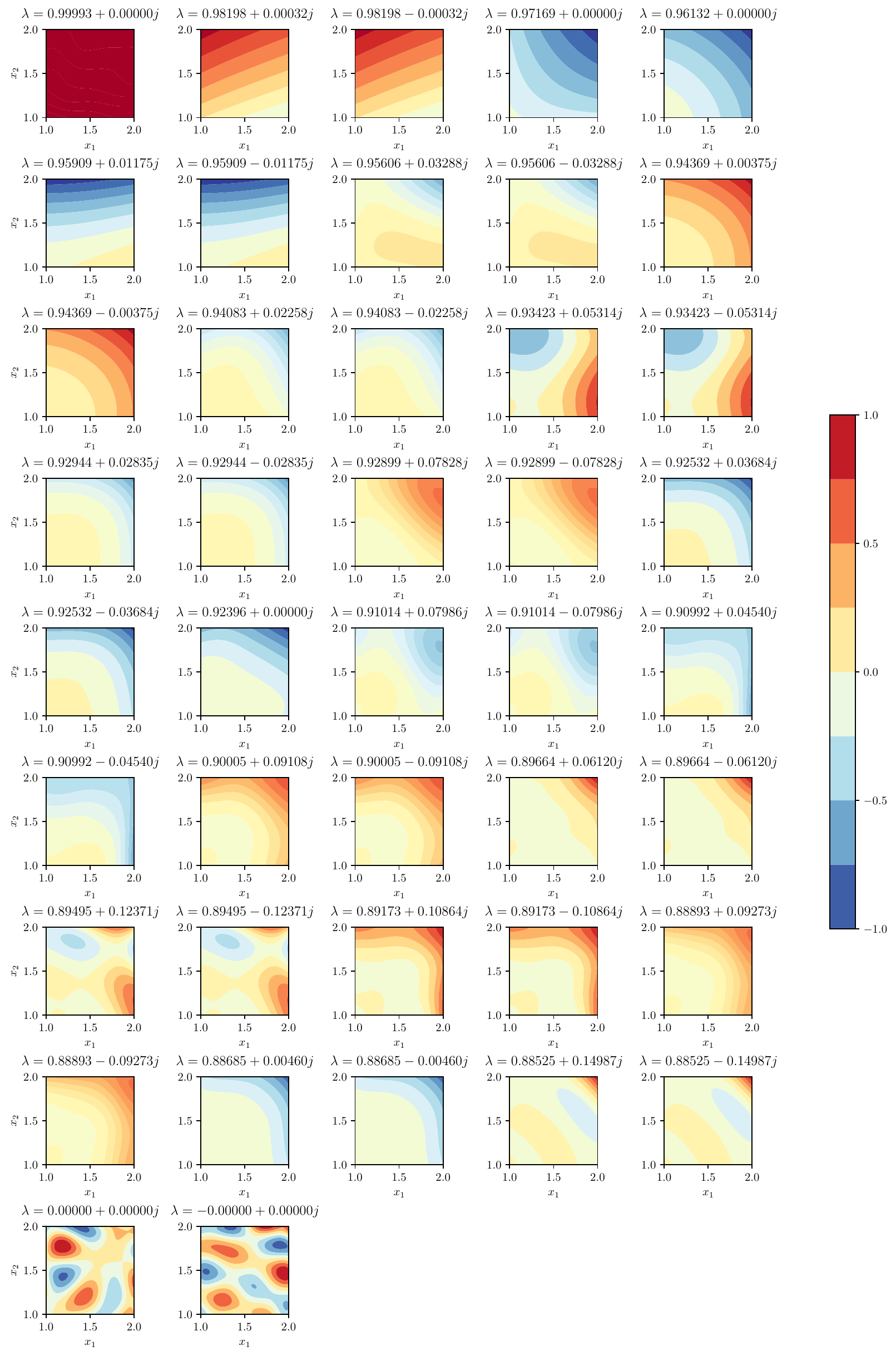}
    \caption{Eigenfunctions for the non-linear system (\ref{eq: nonlin_transformed_eqn}) approximated using EDMD.}
    \label{fig: edmd_eigenfunctions_nonlin_from_lin}
\end{figure}
\begin{figure}[ht]
    \centering
    \includegraphics[width=1\textwidth]{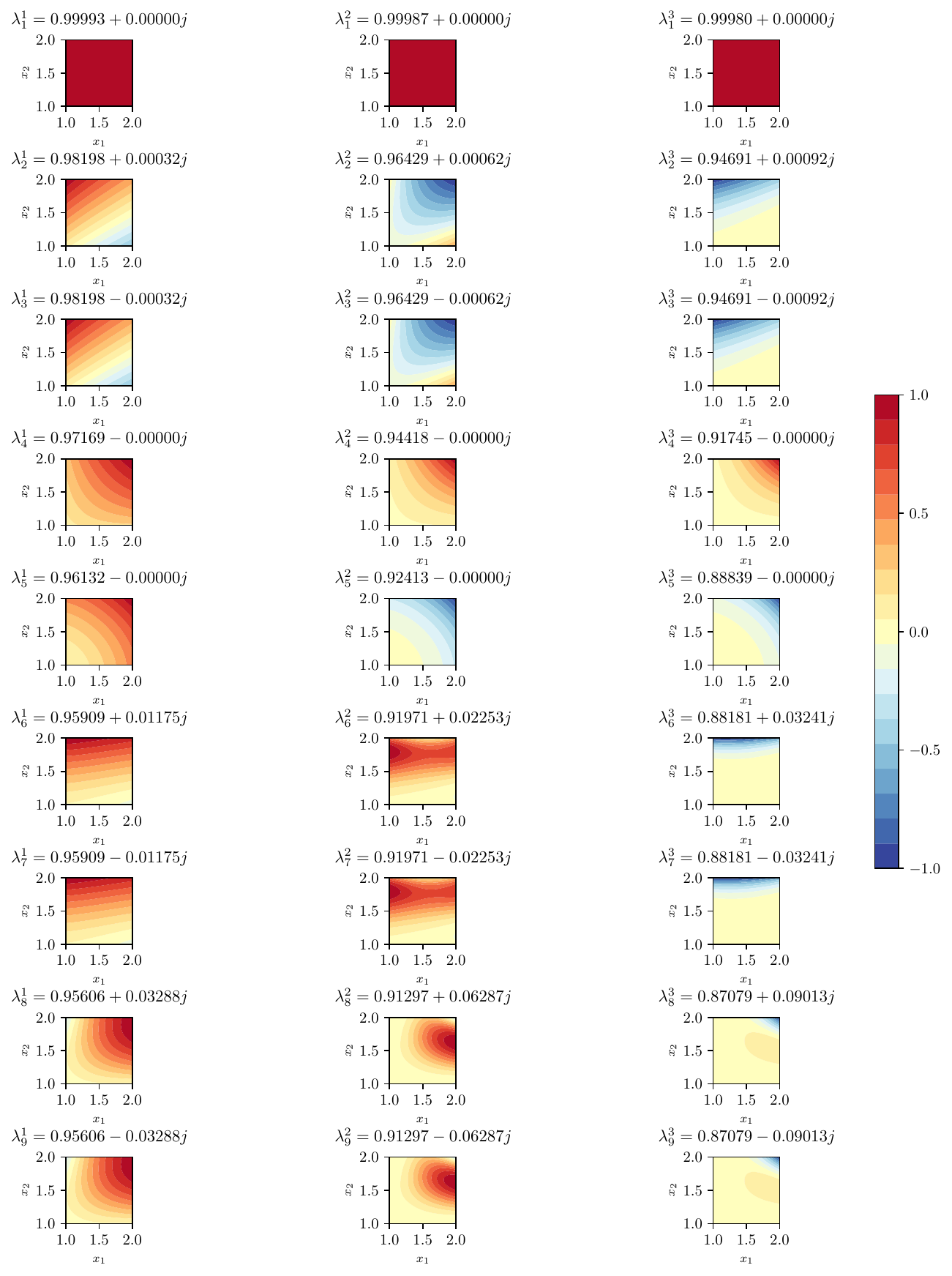}
    \caption{Extended eigenfunctions for the first 9 eigenvalues computed using Algorithm \ref{alg: cont_system_extend_with_bound} for the non-linear system (\ref{eq: nonlin_transformed_eqn}).}
     \label{fig: algorithm_extended_eigenfunctions_diffeo_system}
\end{figure}
\FloatBarrier
\section{
Separatrices for multistable systems}\label{sec:example - duffing}
Consider the unforced Duffing equation
\begin{align}\label{eq-Duffing}
\ddot{x} \, = \, - \delta \dot{x}  - x (\beta + \alpha x^2 ),  
\end{align}
with $\delta=0.5, \, \beta=-1, \, $ and $\alpha = 0.1$. For this setting, there exist two stable spirals at $ (\pm 3.1623, 0)$ and a saddle at the origin. Thus, nearly all initial conditions, excluding those on the stable manifold of the saddle point, will be attracted to one of the spiral equilibria. The stable manifold of the saddle point is plotted as a blue curve, which separates the basins of attraction of the two stable spirals. The unstable manifold of the saddle is depicted in orange, along which the two stable spirals lie (see Fig. \ref{fig:duffing}).
\begin{figure}[H]
    \centering
\includegraphics[width=0.63\textwidth]{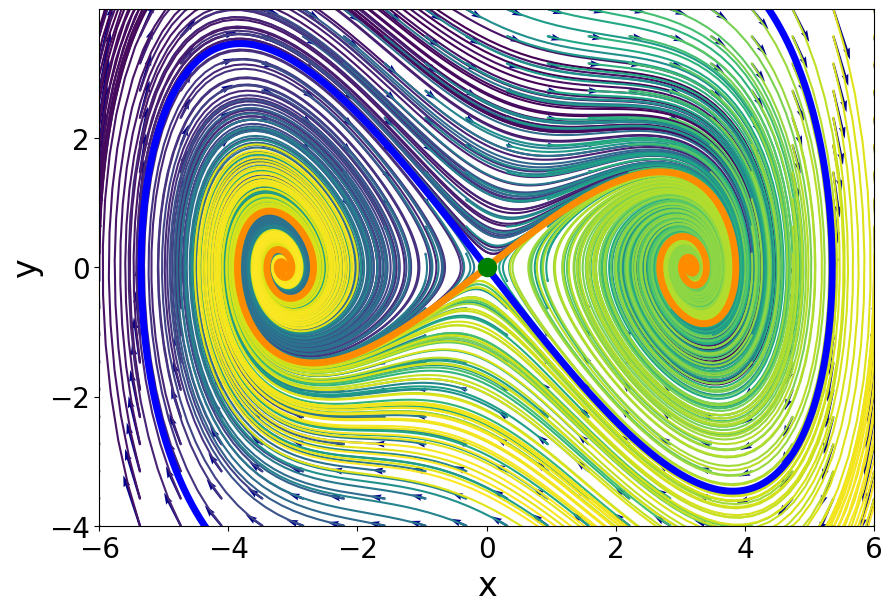}
    \caption{Phase portrait of the unforced Duffing system ($y=\dot{x}$), with two stable spirals and one saddle. The stable and unstable manifolds of the saddle point are depicted as blue and orange curves respectively. }
    \label{fig:duffing}
\end{figure}
\begin{figure}[H]
    \centering
\includegraphics[width=.83\textwidth]{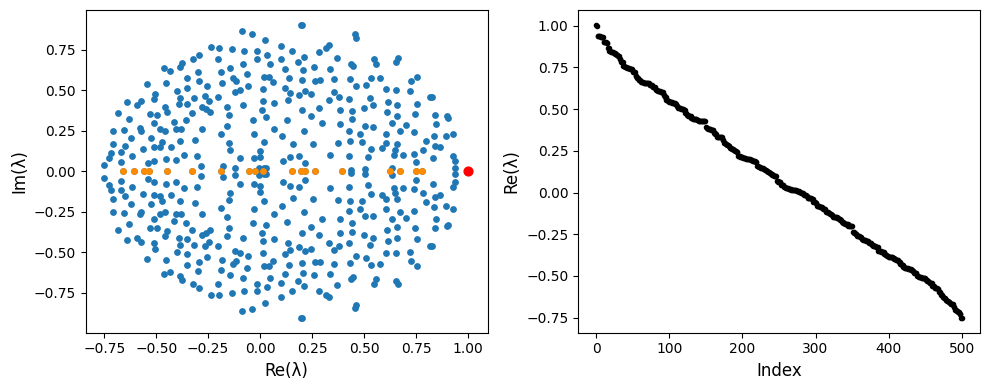}
    \caption{Koopman spectrum computed using EDMD for the unforced Duffing system \eqref{eq-Duffing}. Orange dots indicate real eigenvalues for which $Im(\lambda) = 0$.}
    \label{fig:Duff_spec}
\end{figure}
\begin{figure}[htbp]
    \centering
\includegraphics[width=1\textwidth]{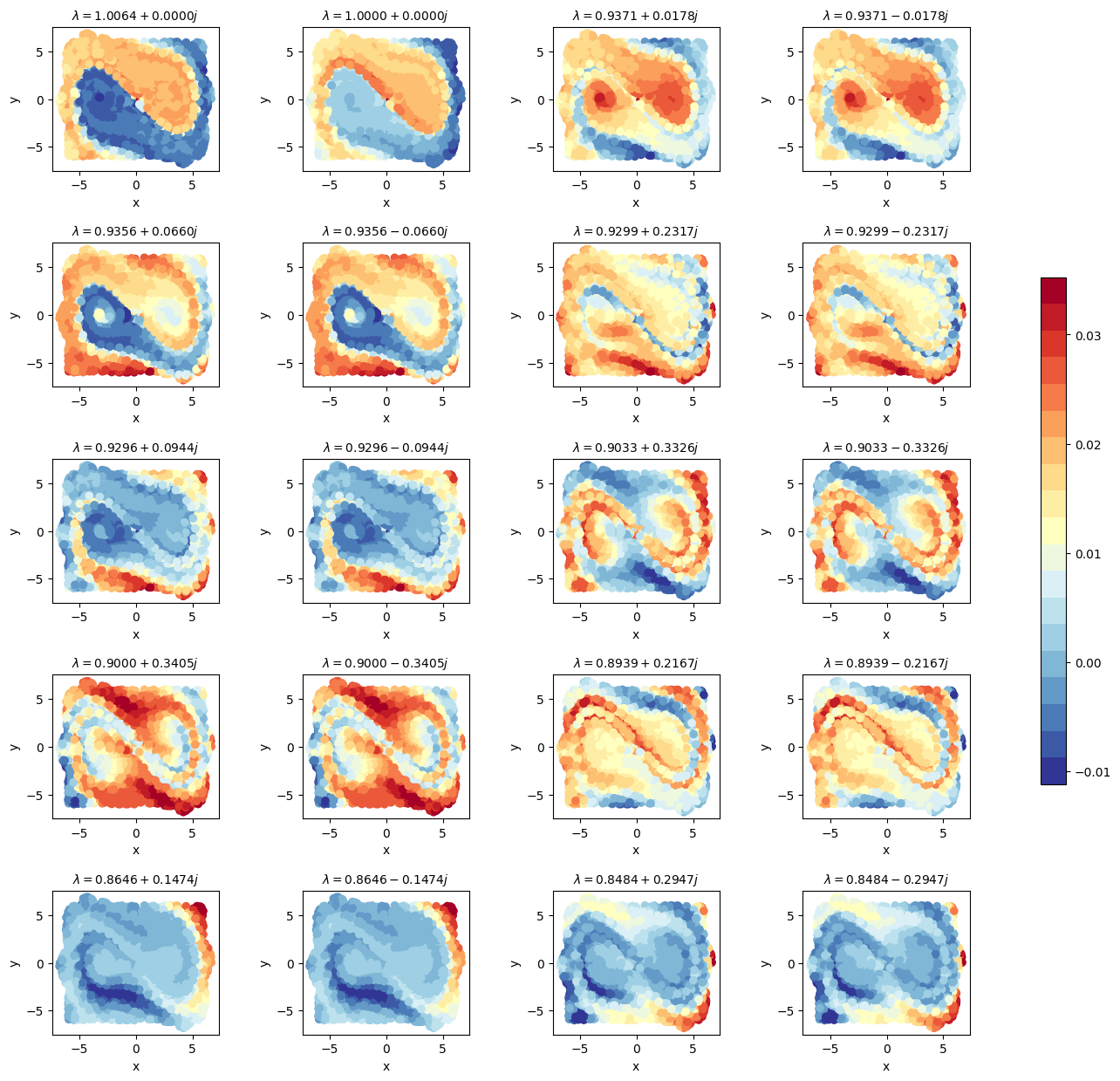}
    \caption{Pseudocolor plots of the first twenty eigenfunctions for the unforced Duffing system \eqref{eq-Duffing} approximated using EDMD.}
    \label{fig:eigenfunction_20}
\end{figure}
 We utilize EDMD to approximate the Koopman eigenfunctions associated with the attractors at $ (\pm 3.1623, 0)$. Following \cite{edmd:2015}, we use a dataset consisting of $3 \times 10^3$ trajectories, each with $11$ samples taken at a sampling interval of 
$\Delta t = 0.25$. The data points are represented as 
$X,Y \in \mathbb{R}^ 
{6\times 10^4}$
 , with initial conditions uniformly distributed over 
$x , \dot{x} = y \in [-6,6]$. We used a dictionary that included $500$ radial basis functions (RBFs). The RBF centers were selected using k-means clustering on the entire data set. Fig. \ref{fig:Duff_spec} and Fig. \ref{fig:eigenfunction_20} illustrate the Koopman spectrum and the first twenty eigenfunctions for the unforced Duffing system\eqref{eq-Duffing}. For the Koopman eigenfunction corresponding to the stable manifold, the magnitude and phase of the eigenfunction parameterize the basin of attraction for each stable spiral. But, the analytical eigenfunctions go to infinity at the boundary between the basins \cite{mauroy-2013}. 
\begin{figure}[H]
    \centering
\includegraphics[width=.55\textwidth]{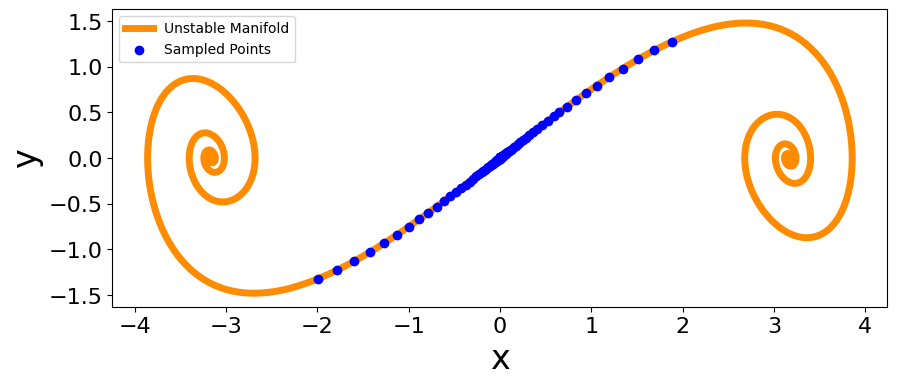}
    \caption{Sampled points $\mathcal{S}$ on the unstable manifold of the saddle points with $(x , y) \in W^u (0,0) \cap [-2,2] \times [-1.33, 1.3]$.}
    \label{fig:sampled_points}
\end{figure}
\begin{figure}[H]
    \centering
\includegraphics[width=1\textwidth]{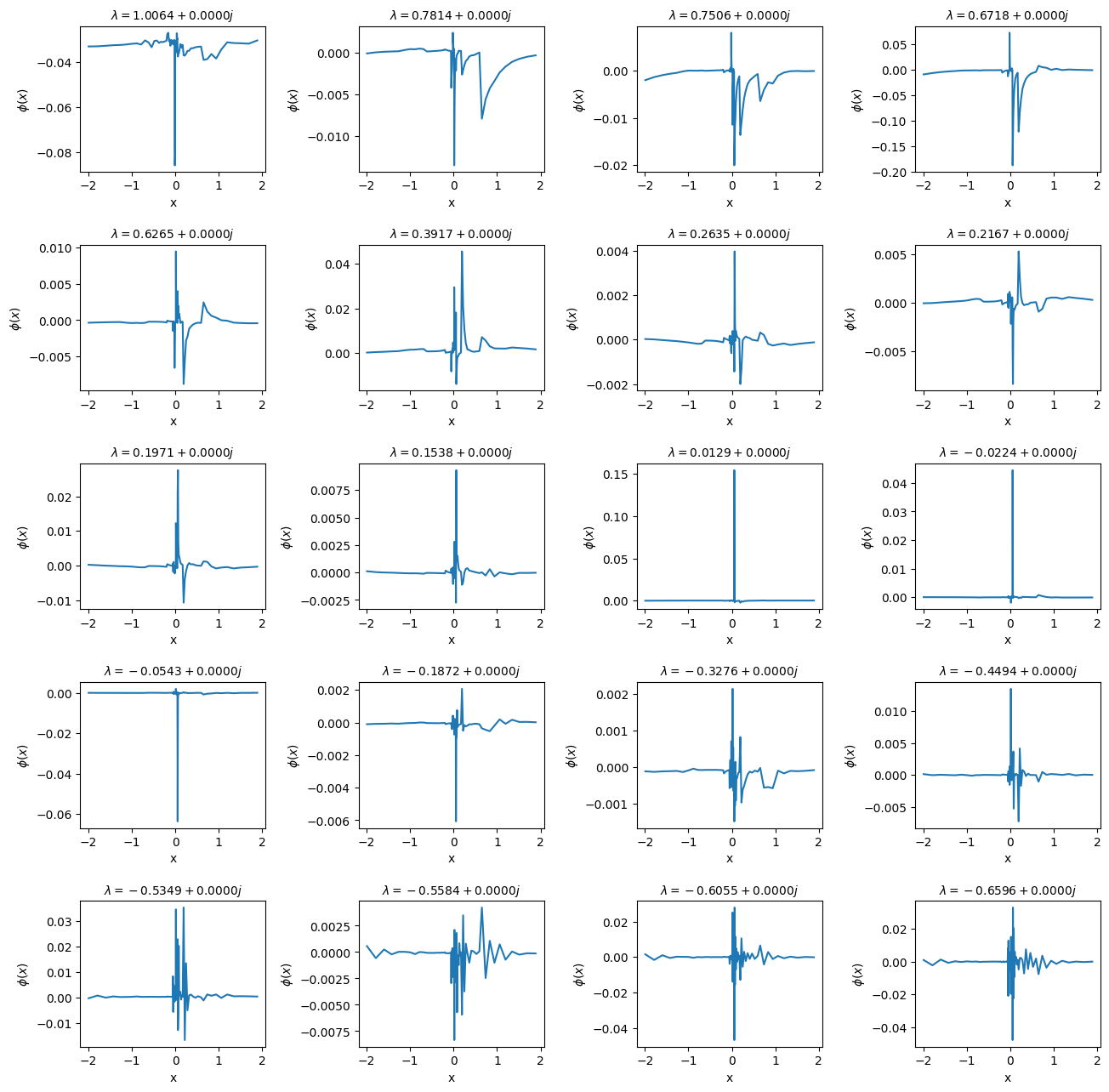}
    \caption{Evaluated eigenfunctions $\phi$ associated with real eigenvalues ($Im(\lambda) = 0$) except for $\lambda = 1$ over the set $\mathcal{S}$.}
    \label{fig:eigenfunction_inf}
\end{figure}
 Theoretically, the eigenfunctions associated with real eigenvalues also approach infinity along the unstable manifold at the saddle point. 
To show this, first we accurately sample the unstable manifold $W^u (0,0)$ around the saddle with $100$ points $ \mathcal{S} \, = \, \{ (x_i , y_i) \, : \, (x_i , y_i) \in W^u (0,0) \cap [-2,2] \times [-1.33, 1.3], \, i= 1, \ldots, 100 \}$ as shown in Fig.  \ref{fig:sampled_points}. Then, we evaluate the previously computed Koopman eigenfunctions (corresponding to the attractors) over the set $\mathcal{S}$. These evaluated eigenfunctions associated with real eigenvalues ( $Im(\lambda) = 0$) except for $\lambda = 1$, are plotted in Fig. \ref{fig:eigenfunction_inf}. 
\section{Additional figures for the bi-stable system \eqref{system:nonlinear2D} in Sect. \ref{sec:example - saddleSystem}}\label{appendix:fig_bistable}
\begin{figure}[htbp]
\centering
\includegraphics[width=.95\textwidth]{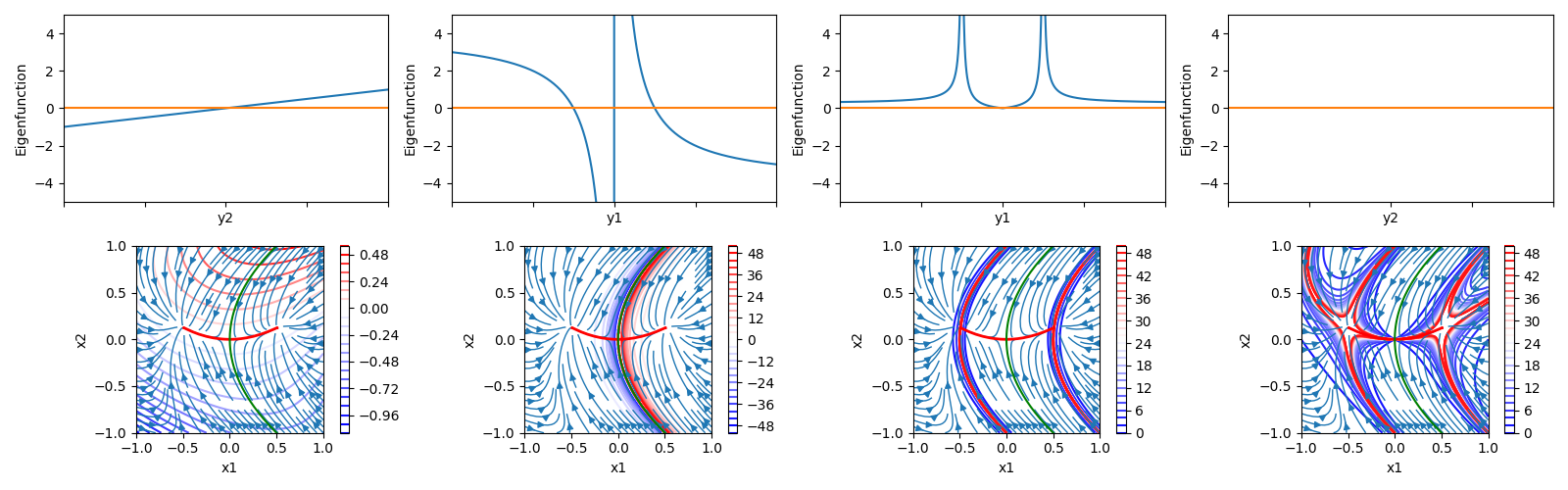}
\caption{Four Koopman eigenfunctions for the system \eqref{system:nonlinear2D}, plotted over the separatrix ($y_1$ coordinate) or the connection between the two steady states through the saddle ($y_2$ coordinate). From the left, the first function is zero on the connection and does not diverge. The second function diverges at the separatrix and is zero on the two steady states. The third function is the inverse of the second. The fourth function is the inverse of the product of the first and the second.}\label{fig_eigenfunctions_saddle_four}
\end{figure}

\end{document}